\documentclass[10pt,a4paper]{amsart}
\usepackage[english]{babel}
\usepackage{amssymb,amsmath,amsthm}
\usepackage{enumitem}
\usepackage{mathtools}

\theoremstyle{plain}
\newtheorem{teo}{Theorem}[section]
\newtheorem{lem}[teo]{Lemma}
\newtheorem{coro}[teo]{Corollary}
\newtheorem{prop}[teo]{Proposition}
\theoremstyle{definition}
\newtheorem{defi}[teo]{Definition}
\numberwithin{equation}{section}

\newcommand{\NN}{\mathbb{N}}
\newcommand{\ZZ}{\mathbb{Z}}
\newcommand{\RR}{\mathbb{R}}
\newcommand{\CC}{\mathbb{C}}

\newcommand{\Rpos}{{\RR^+}}
\newcommand{\Rnon}{{\RR^+_0}}

\newcommand{\CZ}{Calder\'on--Zygmund}
\newcommand{\MH}{Mihlin--H\"ormander}
\newcommand{\CaC}{Carnot--Cara\-th\'eodory}

\newcommand{\condC}{\textup{(C)}}   

\newcommand{\di}{\,\mathrm{d}}    
\newcommand{\rd}{\eta}  

\newcommand{\fdim}{q}    
\newcommand{\vfG}{X}
\newcommand{\vfA}{\breve{X}_0}
\newcommand{\vfN}{\breve{X}}

\newcommand{\step}{S}

\newcommand{\chil}{\mathfrak{C}}  
\newcommand{\desc}{\mathfrak{D}}  

\newcommand{\lie}{\mathfrak}
\DeclareMathOperator{\tr}{tr}
\DeclareMathOperator{\Aut}{Aut}

\DeclareMathOperator{\arctanh}{arctanh}
\DeclareMathOperator{\arccosh}{arccosh}

\DeclareMathOperator{\supp}{supp}

\newcommand{\dist}{\varrho}

\newcommand{\hc}{\mathbf{c}}    

\newcommand{\tc}{\,:\,}

\newcommand{\loc}{\mathrm{loc}}

\newcommand{\acknowledgments}{\section*{Acknowledgments} The authors are members of the Gruppo Nazionale per l'Analisi Matematica, la Probabilit\`a e le loro Applicazioni (GNAMPA) of the Istituto Nazionale di Alta Matematica (INdAM). The first-named author gratefully acknowledges the support of the Deutsche Forschungsgemeinschaft (project MA 5222/2-1). This work was partially supported by the Progetto GNAMPA 2014 ``Analisi Armonica e Geometrica su variet\`a e gruppi di Lie'' and by the Progetto PRIN 2010-2011 ``Variet\`a reali e complesse: geometria, topologia e analisi armonica''.}

\begin{document}

\title
[Spectral multipliers on solvable extensions of stratified groups]
{Spectral multipliers for sub-Laplacians on solvable extensions of stratified groups}

\author[A. Martini]{Alessio Martini}
\address[A. Martini]{School of Mathematics \\ University of Birmingham \\
Edgbaston \\ Birmingham \\ B17 0AN \\ United Kingdom}
\email{a.martini@bham.ac.uk}

\author[A. Ottazzi]{Alessandro Ottazzi}
\address[A. Ottazzi]{School of Mathematics and Statistics \\ University of New South Wales \\ UNSW Sydney NSW 2052 \\ Australia}
\email{a.ottazzi@unsw.edu.au}

\author[M. Vallarino]{Maria Vallarino}
\address[M. Vallarino]{Dipartimento di Scienze Matematiche ``Giuseppe Luigi Lagrange''
\\ Politecnico di Torino\\
Corso Duca degli Abruzzi 24\\ 10129 Torino\\ Italy}
\email{maria.vallarino@polito.it}


\begin{abstract}
Let $G = N \rtimes A$, where $N$ is a stratified group and $A = \mathbb{R}$ acts on $N$ via automorphic dilations. Homogeneous sub-Laplacians on $N$ and $A$ can be lifted to left-invariant operators on $G$ and their sum is a sub-Laplacian $\Delta$ on $G$. We prove a theorem of Mihlin--H\"ormander type for spectral multipliers of $\Delta$. The proof of the theorem hinges on a Calder\'on--Zygmund theory adapted to a sub-Riemannian structure of $G$ and on $L^1$-estimates of the gradient of the heat kernel associated to the sub-Laplacian $\Delta$.
\end{abstract}

\maketitle

\section{Introduction}

Let $N$ be a stratified Lie group of homogeneous dimension $Q \geq 2$. Let $G$ be the semidirect product $N \rtimes A$, where $A=\mathbb R$ acts on $N$ via automorphic dilations. The group $G$ is a solvable extension of $N$ that is not unimodular and has exponential volume growth; see Section \ref{s:NAgeometry} for more details.
For all $p\in [1,\infty]$, let $L^p(G)$ denote the $L^p$ space with respect to a right Haar measure $\mu$ on $G$.

Consider a system $\vfN_1,\dots,\vfN_\fdim$ of left-invariant vector fields on $N$ that form a basis of the first layer of the Lie algebra of $N$ and let $\vfA$ be the standard basis of the Lie algebra of $A$. The vector fields $\vfA$ on $A$ and $\vfN_1,\dots,\vfN_\fdim$ on $N$ can be lifted to left-invariant vector fields $\vfG_0,\vfG_1,\dots,\vfG_\fdim$ on $G$ which generate the Lie algebra of $G$ and define a sub-Riemannian structure on $G$ with associated left-invariant \CaC\ distance $\dist$.

Let $\Delta$ be the left-invariant sub-Laplacian on $G$ defined by
\begin{equation}\label{eq:sub-Laplacian}
\Delta = -\sum_{j=0}^\fdim \vfG_j^2.
\end{equation}
The operator $\Delta$ extends uniquely to a positive self-adjoint operator on $L^2(G)$. For all bounded Borel functions $F : [0,\infty) \to \CC$, the operator $F(\Delta)$ defined via the spectral theorem is left-invariant and bounded on $L^2(G)$ and, by the Schwartz kernel theorem,
\begin{equation}\label{eq:convolutionkernel}
F(\Delta) f = f * k_{F(\Delta)} \qquad \forall f \in L^2(G)\,,
\end{equation}
for some convolution kernel $k_{F(\Delta)}$, which in general is a distribution on $G$. The object of this paper is the multiplier problem for $\Delta$, i.e., the study of sufficient conditions on $F$ which imply the $L^p$-boundedness of $F(\Delta)$ for some $p\neq 2$.

Our main result provides a sufficient condition of \MH\ type for operators of the form $F(L)$ to be bounded on $L^p(G)$ for $1<p<\infty$; endpoint results are also obtained, both of weak type $(1,1)$ and in terms of the Hardy space $H^1(G)$ and bounded mean oscillation space $BMO(G)$ introduced in \cite{V1} (see Section~\ref{s:CZ}).

Let $\psi$ be a function in $C^{\infty}_c(\RR)$, supported in $[1/4,4]$, such that
\begin{equation}\label{eq:sumpsi}
\sum_{j\in\ZZ}\psi(2^{j}\lambda)=1\qquad \forall \lambda \in (0,\infty).
\end{equation}
For all $s \geq 0$ we define $\|F\|_{0,s}$ and $\|F\|_{\infty,s}$ as follows:
\[
\|F\|_{0,s} = \sup_{t<1}\|F(t\cdot)\,\psi(\cdot)\|_{H^s(\RR)},\qquad
\|F\|_{\infty,s} = \sup_{t\geq 1}\|F(t\cdot)\,\psi(\cdot)\|_{H^s(\RR)},
\]
where $H^s(\RR)$ denotes the $L^2$-Sobolev space of order $s$ on $\RR$. We say that a bounded Borel function $F : [0,\infty) \to \CC$ satisfies a \emph{mixed \MH\ condition of order $(s_0,s_{\infty})$} if $\|F\|_{0,s_0}<\infty$ and $\|F\|_{\infty,s_{\infty}}<\infty$.

\begin{teo}\label{thm:moltiplicatori}
Suppose that $s_0>\frac{3}{2}$ and $s_{\infty}>\frac{Q+1}{2}$. If $F$ satisfies a mixed \MH\ condition of order $(s_0,s_{\infty})$, then $F(\Delta)$ extends to an operator of weak type $(1,1)$ and bounded on $L^p(G)$ for all $p \in (1,\infty)$, bounded from $H^1(G)$ to $L^1(G)$ and from $L^{\infty}(G)$ to $BMO(G)$.
\end{teo}

Spectral multiplier theorems for Laplacians and sub-Laplacians have been obtained in many different contexts, so we do not attempt to give a complete account of the existing literature and we restrict our discussion to the works that are more closely related to our result. The interested reader is referred to the cited works and references therein for more details.

It was already known in the literature that, unlike other sub-Laplacians on solvable groups (see, e.g., \cite{christ_spectral_1996,hebisch_sub-Laplacians_2005}), the sub-Laplacian $\Delta$ on the group $G$ has $L^p$-differen\-tiable functional calculus. More precisely, Hebisch \cite{H3} proved that if $F$ is compactly supported and $F\in H^s(\RR)$ for some $s>\frac{Q+5}{2}$, then $F(\Delta)$ is bounded on $L^p(G)$ for all $p \in [1,\infty]$. Mustapha \cite{mustapha_multiplicateurs_1998} proved the same result pushing down the smoothness condition on the multiplier $F$, i.e., requiring that $F\in H^s(\RR)$ for some $s>2$. A further improvement with condition $s>3/2$ is stated in \cite[Theorem 6.1]{HS}. Subsequently Gnewuch \cite{G1} obtained similar results for sub-Laplacians on compact extensions of a class of solvable groups, which strictly include the groups we are considering here.

All these results
 are different from Theorem \ref{thm:moltiplicatori} because they only treat the case of compactly supported multipliers $F$ belonging to a Sobolev space of suitable order and show that, in that case, the convolution kernel $k_{F(\Delta)}$ is integrable on $G$.
Our result instead
is a genuine \MH\ theorem for multipliers $F$ which need not be compactly supported nor have bounded derivatives at $0$. In this case the convolution kernels $k_{F(\Delta)}$ need not be integrable; indeed, for the endpoint values $p=1$ and $p=\infty$ we prove boundedness only in the weak type $(1,1)$ sense and in terms of Hardy and BMO spaces.

Other multiplier theorems on solvable extensions of stratified groups were previously obtained in the literature for distinguished full Laplacians. More precisely, Cowling, Giulini, Hulanicki and Mauceri \cite{cowling_spectral_1994} proved a multiplier theorem for a distinguished Laplacian $L$ on $NA$ groups coming from the Iwasawa decomposition of a semisimple Lie group of arbitrary rank: they showed that if $F \in H^{s_0}_\loc(\RR)$ and $\|F\|_{\infty,s_{\infty}}<\infty$ for suitable orders $s_0,s_\infty$ 
depending on the topological dimension and the pseudodimension of the group, then $F(L)$ is of weak type $(1,1)$ and bounded on $L^p$ for all $p\in (1,\infty)$. An analogous result was then proved by Astengo \cite{A2} for a distinguished Laplacian on Damek--Ricci spaces, i.e., groups of the form $H\rtimes \RR$, where $H$ is a Heisenberg-type group \cite{DR}.

Hebisch and Steger \cite[Theorem 2.4]{HS} improved the results in \cite{cowling_spectral_1994} by proving a genuine \MH\ theorem for spectral multipliers of a distinguished Laplacian $L$ on the group $\RR^Q \rtimes \RR$, which corresponds to the case of real hyperbolic spaces (and coincides with our Theorem \ref{thm:moltiplicatori} in the case $N$ is abelian). Their theorem was generalized in \cite{V2} to a distinguished Laplacian on Damek--Ricci spaces. The results in \cite{HS, V2} hinge on a new abstract \CZ\ theory developed by Hebisch and Steger and $L^1$-estimates of the gradient of the heat kernel associated to $L$.

All the aforementioned results for multipliers of a full Laplacian $L$
make strong use of spherical analysis either on semisimple Lie groups or Damek--Ricci spaces. In particular, on Damek--Ricci spaces, the convolution kernels $k_{F(L)}$ have the property that $m^{-1/2} \, k_{F(L)}$ is radial, where $m$ is the modular function, and moreover an explicit formula for the heat kernel associated to $L$ is known.
These tools are not available for the analysis of the sub-Laplacian $\Delta$ on $G$ (unless $N$ is abelian). So we need new techniques to obtain weighted estimates of the convolution kernels of multipliers of $\Delta$ and to study the horizontal gradient of the heat kernel associated to $\Delta$.
A brief illustration of these techniques and of our strategy of proof follows.

In Section \ref{s:NAgeometry} we obtain a precise description of the left-invariant \CaC\ distance
on $G$ in terms of the analogous distance on $N$. This is done by relating solutions to the Hamilton--Jacobi equations on $G$ and $N$. These equations are analogous to the geodesic equations on Riemannian manifolds. However on sub-Riemannian manifolds there may exist ``strictly abnormal minimizers'', i.e., length-minimizing curves that do not correspond to solutions to the Hamilton--Jacobi equations.
Nevertheless a density result by Agrachev \cite{agrachev_subriemannian_2009} allows us to transfer information from solutions to the Hamilton--Jacobi equations to the corresponding sub-Riemannian distances.

Based on our analysis of distances, in Section \ref{s:CZ} we develop a \CZ\ theory adapted to the sub-Riemannian structure of $G$. More precisely, we show that the metric measure space $(G,\dist,\mu)$ satisfies the axioms of the abstract \CZ\ theory introduced in \cite{HS} and further developed in \cite{V1}. The crucial step is the construction of a suitable family of ``admissible sets'' that play the role that in the classical \CZ\ theory on spaces of homogeneous type would be played by balls or ``dyadic cubes'' (cf.\ \cite{C}). 
In this way, when we study spectral multipliers of the sub-Laplacian $\Delta$, we can use the theorems for singular integral operators proved in \cite{HS} for the boundedness of type $(1,1)$ and those contained in \cite{V1} for the boundedness on Hardy and BMO spaces.

In Section \ref{s:sub-Laplacian} we focus on the properties of $\Delta$ and its functional calculus. In particular 
Section \ref{subs:heat} is devoted to an $L^1$-estimate of the horizontal gradient of the heat kernel associated to $\Delta$ at any real time. This estimate is well-known (in much greater generality) for small time, but appears to be new for large time (and nonabelian $N$).
Our proof is based on a formula that relates the sub-Riemannian heat kernels on $G$ and $N$; this relation was already used in \cite{mustapha_multiplicateurs_1998,gnewuch_differentiable_2006} to estimate the heat kernel on $G$ at complex time $1+i\tau$, $\tau \in \RR$. 

Another important consequence of the relation between heat kernels on $G$ and $N$ is discussed in Section \ref{subs:plancherel}. It turns out that, for all multipliers $F$, the $L^2$-norm of the convolution kernel $k_{F(\Delta)}$ on $G$ coincides with the $L^2$-norm of the convolution kernel $k_{F(\tilde\Delta)}$ on the real hyperbolic space $\tilde G = \RR^Q \rtimes \RR$, where $\tilde\Delta$ is a full Laplacian on $\tilde G$.
In fact it is even possible to estimate weighted $L^2$-norms of $k_{F(\Delta)}$ on $G$ by weighted $L^2$-norms of $k_{F(\tilde\Delta)}$ on $\tilde G$, where spherical analysis can be applied. This crucial observation is already contained, with a different proof, in \cite{HEB}.

Finally in Section \ref{s:theorem} we combine all these ingredients to prove Theorem \ref{thm:moltiplicatori}.

A natural question is if the smoothness condition
 $s_\infty > (Q+1)/2$ 
on the multiplier
in Theorem \ref{thm:moltiplicatori} is sharp.
In fact, via transplantation (cf.\ \cite{kenig_divergence_1982}), Theorem \ref{thm:moltiplicatori} implies a similar theorem for a homogeneous sub-Laplacian on the nilpotent contraction $N \times A$ of $G$, with a smoothness condition of order $(Q+1)/2$.
This is just a particular case of the multiplier theorem of Christ \cite{christ_multipliers_1991} and Mauceri and Meda \cite{mauceri_vectorvalued_1990} on stratified groups, because $Q+1$ is the homogeneous dimension of $N\times A$.
If $N$ is abelian, then the transplanted result is sharp and a fortiori the condition $s_\infty > (Q+1)/2$ in Theorem \ref{thm:moltiplicatori} is sharp.
However, for many nonabelian stratified groups $N$ the transplanted result is not sharp: in fact, in several cases, it is possible to push down the smoothness condition to half the topological dimension of the group \cite{hebisch_multiplier_1993,MueS,martini_heisenbergreiter,MMue}. For this reason it might be expected that the smoothness condition $s_\infty > (Q+1)/2$ in Theorem \ref{thm:moltiplicatori} could also be pushed down, at least for some nonabelian $N$.

Recently the second and third named authors, extending a result in \cite{hebisch_spectral_2005}, have proved a multiplier theorem for some Laplacians with drift on Damek--Ricci spaces \cite{OV}; part of the proof of their result hinges on a \MH\ type theorem for a distinguished Laplacian without drift. Inspired by \cite{OV} we think that Theorem \ref{thm:moltiplicatori} could be an ingredient to prove a multiplier theorem for sub-Laplacians with drift on the solvable groups considered here. We recall that among these sub-Laplacians with drift there is the ``intrinsic hypoelliptic Laplacian'' associated with the sub-Riemannian structure on $G$ (see \cite{ABGR}).

Let us fix some notation that will be used throughout.
$\Rpos$ and $\Rnon$ denote the open and closed positive half-lines in $\RR$ respectively.
$\bigcup \mathcal{R}$ denotes the union of a family of sets $\mathcal{R}$, i.e., $\bigcup \mathcal{R} = \bigcup_{R \in \mathcal{R}} R$.
The letter $C$ and variants such as $C_s$ denote finite positive constants that may vary from place to place.
Given two expressions $A$ and $B$, $A\lesssim B$ means that there exists a finite positive constant $C$ such that $ A\le C \,B $. Moreover $A \sim B$ means $A \lesssim B$ and $B \lesssim A$.

\section{Solvable extensions of stratified groups}\label{s:NAgeometry}
In this section we shall introduce the class of Lie groups that we study in the sequel and recall their main properties. In particular, we shall discuss their metric properties in Subsection \ref{subs:metric} and some useful integral formulas in Subsection \ref{subs:integralradial}.

\subsection{Stratified groups and their extensions}\label{subs:NA}

Let $N$ be a stratified group. In other words, $N$ is a simply connected Lie group, whose Lie algebra $\lie{n}$ is endowed with a derivation $D$ such that the eigenspace of $D$ corresponding to the eigenvalue $1$ generates $\lie{n}$ as a Lie algebra. In particular the eigenvalues of $D$ are positive integers $1,\dots,\step$ and $\lie{n}$ is the direct sum of the eigenspaces of $D$, which are called layers: the $j$th layer corresponds to the eigenvalue $j$. Moreover $\lie{n}$ is $\step$-step nilpotent, where $\step$ is the maximum eigenvalue.

The exponential map $\exp_N : \lie{n} \to N$ is a diffeomorphism and provides global coordinates for $N$, that shall be used in the sequel without further mention. Any chosen Lebesgue measure on $\lie{n}$ is then a left and right Haar measure on $N$. Let us fix such a measure and write $|E|$ for the measure of a measurable subset $E \subset N$.

The formula $\delta_t = \exp((\log t) D)$ defines a family of automorphic dilations $(\delta_t)_{t>0}$ on $N$. For all measurable sets $E \subset N$ and $t > 0$, $|\delta_t E| = t^Q |E|$, where $Q = \tr D$ is the homogeneous dimension of $N$. Note that $Q \geq d$, where $d = \dim \lie{n}$ is the topological dimension of $N$, and in fact $Q = d$ if and only if $\step=1$, i.e., if and only if $N$ is abelian. Note moreover that, if $Q = 1$, then $N \cong \RR$. In the following we shall assume that $Q \geq 2$, since the case $Q=1$ has already been treated in \cite{HS}.

Let $A = \RR$, considered as an abelian Lie group. Again we identify $A$ with its Lie algebra $\lie{a}$. Then $A$ acts on $N$ by dilations, that is, we have a homomorphism $A \ni u \mapsto \delta_{e^u} \in \Aut(N)$ and we can define the corresponding semidirect product $G = N \rtimes A$, with operations
\[
(z,u) \cdot (z',u') = (z \cdot e^{uD} z', u+u'), \qquad (z,u)^{-1} = (-e^{-uD} z,-u)
\]
and identity element $0_G=(0_N,0)$. The Lie algebra $\lie{g}$ of $G$ is then canonically identified \cite[\S 3.14-3.15]{varadarajan_lie_1974} with the semidirect product of Lie algebras $\lie{n} \rtimes \lie{a}$, where
\[
[(z,u),(z',u')] = ([z,z']+uD z'-u'Dz,0).
\]

The group $G$ is not nilpotent, but is a solvable Lie group of topological dimension $d+1$.
The left and right Haar measures $\mu_\ell$ and $\mu$ on $G$ are given by
\[
\di\mu_\ell(z,u) = e^{-Qu} \di z \di u  \qquad \di\mu(z,u) = \di z \di u
\]
\cite[\S (15.29)]{hewitt_abstract_1979} and the modular function $m$ is given by $m(z,u) =  e^{-Qu}$. In particular $G$ is not unimodular and has exponential volume growth \cite[Lemme I.3]{guivarch_croissance_1973}. In the following, unless otherwise specified, the right Haar measure $\mu$ will be used to define Lebesgue spaces $L^p(G) = L^p(G,\di \mu)$ on $G$ and $\|f\|_p$ will denote the $L^p(G)$-norm of a function $f$ on $G$.

\subsection{Metric structure and geodesics}\label{subs:metric}

Consider a system $\vfN_1,\dots,\vfN_\fdim$ of left-invariant vector fields on $N$ that form a basis of the first layer of $\lie{n}$. These vector fields provide a global frame for a subbundle $HN$ of the tangent bundle $TN$ of $N$, called the horizontal distribution. Since $N$ is stratified, the first layer generates $\lie{n}$ as a Lie algebra and consequently the horizontal distribution is bracket-generating.

Let $g^N$ be the left-invariant sub-Riemannian metric on the horizontal distribution of $N$ which makes $\vfN_1,\dots,\vfN_\fdim$ into an orthonormal basis. By means of the metric $g^N$ we can define the length of horizontal curves on $N$ (i.e., absolutely continuous curves $\gamma : [a,b] \to N$ whose tangent vector $\dot\gamma(t)$ lies in the horizontal distribution for almost all $t \in [a,b]$) by integrating the $g^N$-norm of the tangent vector. The \CaC\ distance $\dist^N$ on $N$ associated to $g^N$ is then defined by
\[
\dist^N(z,z') = \inf \{ \text{lengths of horizontal curves joining $z$ to $z'$} \}
\]
for all $z,z' \in N$. Since the horizontal distribution is bracket-generating, the distance $\dist^N$ is finite and induces on $N$ the usual topology, by the Chow--Rashevskii theorem. Moreover, since $\vfN_1,\dots,\vfN_\fdim$ are left-invariant and belong to the first layer, the distance $\dist^N$ is left-invariant and homogeneous with respect to the automorphic dilations $\delta_t$. For every $z_0\in N$ and $r> 0$ we denote by $B_N(z_0,r)$ the ball in $N$ centered at $z_0$ of radius $r$, i.e., $B_N(z_0,r)=\{z\in N: \dist^N(z,z_0)<r\}$. Then
\[
|B_N(z_0,r)| =  r^Q |B_N(0_N,1)| \qquad \forall z_0\in N,\,\forall r> 0.
\]

Let $\vfA = \partial_u$ be the canonical basis of $\lie{a}$. The vector fields $\vfA$ on $A$ and $\vfN_1,\dots,\vfN_\fdim$ on $N$ can be lifted to left-invariant vector fields on $G$ given by
\[
\vfG_0|_{(z,u)} = \vfA|_z = \partial_u,  \qquad \vfG_j|_{(z,u)} = e^{u} \vfN_j|_z \qquad \text{ for $j=1,\dots,\fdim.$}
\]
Analogously as above, the system $\vfG_0,\dots,\vfG_\fdim$ generates the Lie algebra $\lie{g}$ and defines a sub-Riemannian structure on $G$ with associated left-invariant \CaC\ distance $\dist$. For all $(z_0,u_0)\in G$ and $r> 0$ we denote by $B_{\varrho}\big((z_0,u_0),r\big)$ the ball in $G$ centered at $(z_0,u_0)$ with radius $r$ with respect to the distance $\dist$.

We shall give a more precise description of the distance $\dist$ and precise asymptotics for the volume of balls by means of geodesics. Note that the characterization of length-minimizing curves in sub-Riemannian geometry is more complicated than in the Riemannian case, because a length-minimizing curve need not correspond to a solution of the Hamilton--Jacobi equations associated to the metric (see, e.g., \cite{Montgomery} for an insight). However, by means of a density result of Agrachev \cite{agrachev_subriemannian_2009}, we will be able to characterize the distance $\dist$ by studying the solutions of the Hamilton--Jacobi equations on $N$ and $G$.

The sub-Riemannian metric $g^N$ determines a dual metric $(g^N)^*$ on the cotangent bundle $T^*N$ of $N$. When $\step > 1$, $(g^N)^*$ is degenerate: its kernel at each point of $N$ is the annihilator of the horizontal distribution. If $N$ is identified as a manifold with the vector space $\lie{n}$ via the exponential map (see Section \ref{subs:NA}), then,
for all $z \in N$, the tangent space $T_z N$ at $z$ is identified with $\lie{n}$ and the cotangent space $T^*_z N$ is identified with $\lie{n}^*$. Let us in turn identify $\lie{n}^*$ with $\lie{n}$ by choosing an inner product $\langle \cdot, \cdot \rangle$ on $\lie{n}$
and let us fix orthonormal coordinates on $\lie{n}$. Then
\[
(g^N)^*_z(\zeta,\zeta') = \langle M_z \zeta, \zeta' \rangle,
\]
where $M_z : \lie{n} \to \lie{n}$ is a symmetric linear map depending smoothly on $z \in N$; moreover $H_z N$ is the range of $M_z$, the restriction $M_z|_{H_z N} : H_z N \to H_z N$ is invertible and
\[
g^N_z(Z,Z') = \langle (M_z|_{H_z N})^{-1} Z, Z' \rangle.
\]
In the chosen coordinates, the Hamilton--Jacobi equations associated to $g^N$ read
\begin{equation}\label{eq:HJ_N}
\dot z_j = \frac{\partial H^N}{\partial \zeta_j}, \qquad \dot \zeta_j = -\frac{\partial H^N}{\partial z_j}
\end{equation}
($j=1,\dots,d$), where the Hamiltonian $H^N : T^* N \to \RR$ is given by
\[
H^N(z,\zeta) = \frac{1}{2} (g^N)^*_z(\zeta,\zeta) = \frac{1}{2} \langle M_z \zeta, \zeta \rangle.
\]
A solution $(z,\zeta) : I \to T^* N$ to the Hamilton--Jacobi equations \eqref{eq:HJ_N}, where $I \subset \RR$ is an interval, will be called an \emph{HJ-curve} on $N$. It is known that the projection to $N$ of such a curve, namely $z : I \to N$, is horizontal and locally length-minimizing; moreover $z$ has constant speed, since $g^N_z(\dot z,\dot z) = 2 H^N(z,\zeta)$ is constant along the HJ-curve $(z,\zeta)$. We define the \emph{length} of an HJ-curve as the length of its projection. Analogously, we say that an HJ-curve joins two points on $N$ if its projection does.

Note that, if $\step \leq 2$, then all length-minimizing horizontal curves on $N$ are ``normal minimizers'', i.e., projections of HJ-curves (see, e.g., the argument after \cite[Theorem 4]{monti_regularity_2014}). However on higher-step groups $N$ there may exist ``strictly abnormal length-minimizers'' \cite{gole_note_1995}, that is, length-minimizers that are not projections of HJ-curves.

An analogous discussion can be conducted on $G$. If $G$ is identified as a manifold with the vector space $\lie{n} \times \lie{a}$ via the map $\lie{n} \times \lie{a} \ni (z,u) \mapsto (\exp_N(z),u) \in G$ (as in Section \ref{subs:NA}),
the left-invariant sub-Riemannian metric $g$ on the horizontal distribution of $TG$ is given by
\[
g_{(z,u)}((Z,U),(Z',U')) = e^{-2u} g^N_z(Z,Z') + U U'.
\]
Hence the dual metric $g^*$ on the cotangent bundle $T^* G$ of $G$ is
\[
g^*_{(z,u)}((\zeta,\nu),(\zeta',\nu')) = e^{2u} (g^N)^*_z(\zeta,\zeta') + \nu \nu'
\]
and the Hamilton--Jacobi equations on $G$ read
\begin{equation}\label{eq:HJ_G}\begin{aligned}
\dot z_j &= \frac{\partial H}{\partial \zeta_j}, & \dot \zeta_j &= -\frac{\partial H}{\partial z_j},\\
\dot u &= \frac{\partial H}{\partial \nu}, & \dot \nu &= -\frac{\partial H}{\partial u}
\end{aligned}\end{equation}
($j=1,\dots,d$), where the Hamiltonian $H : T^*G \to \RR$ is given by
\[
H(z,u,\zeta,\nu) = \frac{1}{2} g^*_{(z,u)}((\zeta,\nu),(\zeta,\nu)) = \frac{1}{2} \left( e^{2u} \langle M_z \zeta, \zeta \rangle + \nu^2 \right).
\]
A solution $(z,u,\zeta,\nu) : I \to T^* G$ to \eqref{eq:HJ_G} will be called an HJ-curve on $G$.

We now look for HJ-curves on $G$ of the form $(z,u,\zeta,\nu) = (z^N \circ v, u, \zeta^N \circ v, \nu)$ where $(z^N,\zeta^N)$ is an HJ-curve on $N$ and $v$ is a suitable change of variables. By plugging these expressions in the Hamilton--Jacobi equations for $G$ and using the fact that $(z^N,\zeta^N)$ satisfies the Hamilton--Jacobi equations for $N$, we obtain the following result.

\begin{lem}\label{lem:geodesic_N_G}
Let $(z^N,\zeta^N)$ be an HJ-curve on $N$. Then $(z^N \circ v, u, \zeta^N \circ v, \nu)$ is an HJ-curve on $G$ provided the functions $v,u,\nu$ satisfy the following conditions:
\begin{equation}\label{eq:reparametrization}
\dot v = e^{2u}, \qquad \dot u = \nu, \qquad \dot \nu = -2 H^N_0 e^{2u},
\end{equation}
where $H^N_0$ is the constant value of $H^N$ along $(z^N,\zeta^N)$. Moreover, $H_0^N$ is related to the constant value $H_0$ of $H$ along $(z^N \circ v, u, \zeta^N \circ v, \nu)$ as follows:
\[
H_0 = e^{2u} H_0^N + \nu^2/2.
\]
\end{lem}

This leads us to the following definition.

\begin{defi}\label{dfn:associated}
We say that an HJ-curve $(z^N,\zeta^N) : J \to T^* N$ on $N$ and an HJ-curve $(z,u,\zeta,\nu) : I \to T^* G$ on $G$ are \emph{associated} if there exists a diffeomorphism $v : I \to J$ such that $z = z^N \circ v$, $\zeta = \zeta^N \circ v$, and $(v,u,\nu) : I \to \RR^3$ solves \eqref{eq:reparametrization}.
\end{defi}

The Cauchy problem for the autonomous system of equations \eqref{eq:reparametrization} is solved as follows.

\begin{lem}\label{lem:variablechange}
Suppose that $u_0,\nu_0,H_0^N \in \RR$ and $H_0 \geq 0$. In the case $H^N_0 > 0$, the maximal solution $(v,u,\nu)$ to \eqref{eq:reparametrization} with initial data
\begin{equation}\label{eq:initialdata}
v(0)=0, \quad u(0) = u_0, \quad \nu(0) = \nu_0
\end{equation}
is given by
\begin{align*}
v(t) &= \frac{1}{2H_0^N} (\omega \tanh(\omega(t-t_*)) + \nu_0),\\
u(t) &= u_* - \log \cosh(\omega(t-t_*)), \\
\nu(t) &= -\omega \tanh(\omega(t-t_*)),
\end{align*}
where $u_*,t_*,\omega$ are determined so to satisfy the equations and the initial conditions:
\[
\omega = \sqrt{\nu_0^2 + 2H_0^N e^{2u_0}}, \quad u_* = \log \frac{\omega}{\sqrt{2H_0^N}}, \quad t_* = \frac{1}{\omega} \arctanh \frac{\nu_0}{\omega}.
\]
In the case $H^N_0 =0$, the solution with initial data \eqref{eq:initialdata} is given by
\[
v(t) = \begin{cases}
e^{2u_0} \frac{e^{2\nu_0 t}-1}{2\nu_0} & \text{if $\nu_0 \neq 0$,} \\
e^{2u_0} t & \text{if $\nu_0 = 0$,}
\end{cases}, \qquad u(t) = u_0+\nu_0 t, \qquad \nu(t) = \nu_0.
\]
All these solutions $(v,u,\nu)$ are defined globally in time and $v$ is always an increasing diffeomorphism onto its image. Moreover, for all $u_1 \in \RR$ and $T > 0$, the following conditions are equivalent:
\begin{enumerate}[label=(\roman*)]
\item\label{en:variablechange1} $T$ is in the range of $v$ and $u(v^{-1}(T)) = u_1$;
\item\label{en:variablechange2} $\nu_0 = (2T)^{-1} (e^{2u_1}-e^{2u_0}) + H_0^N T$.
\end{enumerate}
\end{lem}
\begin{proof}
It is not difficult to check that the above formulas give solutions to \eqref{eq:reparametrization} with initial data \eqref{eq:initialdata}. Since they are globally defined in time, they must be the maximal solutions, and $v$ is an increasing diffeomorphism onto its image because $\dot v = e^{2u} > 0$. It remains to show the equivalence of the conditions \ref{en:variablechange1} and \ref{en:variablechange2}; we shall just consider the case $H_0^N > 0$, the other case being similar and easier.

Simple manipulations of the above formulas for $u$ and $v$, also by means of the identity $1/\cosh^2 x = 1-\tanh^2 x$, yield
\[
2 u(t) = \log \left ( \frac{\omega^2}{2H_0^N} \left(1-\frac{(2H_0^N v(t) - \nu_0)^2}{\omega^2}\right) \right),
\]
that is,
\begin{equation}\label{eq:rel_u_v}
e^{2u(t)} = e^{2u_0} + 2 v(t) \, (\nu_0 -H_0^N v(t)).
\end{equation}
In particular, if there exists $t \in \RR$ with $u(t) = u_1$ and $v(t) = T$, then by solving \eqref{eq:rel_u_v} for $\nu_0$ we obtain \ref{en:variablechange2} above. Vice versa, if \ref{en:variablechange2} holds, then
\[
2 H_0^N T^2 -2T\nu_0 = e^{2u_0} - e^{2u_1} < e^{2u_0},
\]
hence
\[
(2H_0^N T - \nu_0)^2 < 2 H_0^N e^{2u_0} + \nu_0^2 = \omega^2.
\]
Because of the explicit formula for $v$, this means that $T$ belongs to the range of $v$, so $v(t) = T$ for some $t \in \RR$ and \eqref{eq:rel_u_v} together with \ref{en:variablechange2} yields $u(t) = u_1$.
\end{proof}

From the above explicit solution we derive several consequences. First, we can construct HJ-curves on $G$ starting from HJ-curves on $N$.

\begin{prop}\label{prp:geodesic_N_G}
Suppose that $T >0$, $(z^N,\zeta^N) : [0,T] \to T^* N$ is an HJ-curve on $N$ and $u_0,u_1 \in \RR$. Then there exists an HJ-curve on $G$ associated to $(z^N,\zeta^N)$ that joins $(z^N(0),u_0)$ to $(z^N(T),u_1)$.
\end{prop}
\begin{proof}
Set $\nu_0 = (2T)^{-1} (e^{2u_1}-e^{2u_0}) + H_0^N T$. If $(v,u,\nu)$ is the maximal solution to \eqref{eq:reparametrization} with initial data \eqref{eq:initialdata}, then, by Lemma \ref{lem:variablechange}, $T$ is in the range of $v$ and $u(v^{-1}(T)) = u_1$. Therefore, by Lemma \ref{lem:geodesic_N_G}, $(z^N \circ v,u,\zeta^N \circ v, \nu) : [0,v^{-1}(T)] \to T^* G$ is an HJ-curve on $G$ associated to $(z^N,\zeta^N)$, which clearly joins $(z^N(0),u_0)$ to $(z^N(T),u_1)$.
\end{proof}

Vice versa, HJ-curves on $G$ determine HJ-curves on $N$.

\begin{prop}\label{prp:geodesic_G_N}
Every HJ-curve on $G$ is associated to an HJ-curve on $N$.
\end{prop}
\begin{proof}
Let $(z,u,\zeta,\nu) : I \to T^* G$ be an HJ-curve on $G$. Without loss of generality we may assume that $0 \in I$. Let $(z^N,\zeta^N) : J \to \RR$ be the maximal solution to the Hamilton--Jacobi equations \eqref{eq:HJ_N} on $N$ with initial data $z^N(0) = z(0)$, $\zeta^N(0) = \zeta(0)$. Let $H^N_0$ be the constant value of $H^N$ along $(z^N,\zeta^N)$ and define
\[
u_0 = u(0), \qquad \nu_0 = \nu(0).
\]
Let $(v,\tilde u,\tilde\nu)$ be the solution to \eqref{eq:reparametrization} with initial data \eqref{eq:initialdata} given by Lemma \ref{lem:variablechange}. Then $(z,u,\zeta,\nu)$ and $(z^N \circ v,\tilde u,\zeta^N \circ v, \tilde \nu)$ are both solutions to \eqref{eq:HJ_G} with the same initial condition and in particular (by uniqueness of solutions to ODEs) they must coincide on the intersections of their intervals of definition.

In order to conclude, it will be sufficient to show that $I$ is contained in the domain $\tilde I$ of $(z^N \circ v,\tilde u,\zeta^N \circ v, \tilde \nu)$. Note that the solution $(v,\tilde u,\tilde \nu)$ to \eqref{eq:reparametrization} given by Lemma \ref{lem:variablechange} is defined globally in time, $v : \RR \to v(\RR)$ is an increasing diffeomorphism and $\tilde I = v^{-1}(J)$ is open. Therefore, if $I$ is not contained in $\tilde I$, then there is a (nonzero) element $t_0 \in I$ of minimum modulus that does not belong to $\tilde I$. Assume, without loss of generality, that $t_0>0$. Then $v(t_0)$ does not belong to the domain $J$ of $(z^N,\zeta^N)$, but $[0,v(t_0)) \subset J$. The equation
\[
(z^N(v(t)),\zeta^N(v(t))) = (z(t),\zeta(t)),
\]
valid for all $t \in [0,t_0)$, and the fact that $v$ is a diffeomorphism, show that
\[
\lim_{\tau \to v(t_0)} (z^N(\tau),\zeta^N(\tau)) = (z(t_0),\zeta(t_0)).
\]
This contradicts the fact that $(z^N,\zeta^N)$ is a maximal solution to \eqref{eq:HJ_N}.
\end{proof}

Finally, there is a relation between lengths of associated HJ-curves.

\begin{prop}
Let $I \subset \RR$ be a compact interval. Let $(z,u,\zeta,\nu) : I \to T^* G$ be an HJ-curve on $G$ of length $L$, which is associated to an HJ-curve on $N$ of length $L^N$. Let $u_0$ and $u_1$ be the values of $u$ at the endpoints of $I$. Then
\begin{equation}\label{eq:lengths}
\cosh L = \frac{1+e^{2(u_1-u_0)}+(e^{-u_0} L^N)^2}{2e^{u_1-u_0}}.
\end{equation}
\end{prop}
\begin{proof}
Let $(z^N,\zeta^N) : J \to T^* N$ be the associated HJ-curve on $N$ and $v : I \to J$ be the diffeomorphism as in Definition \ref{dfn:associated}. Without loss of generality we may assume that $I = [0,\tau]$ with $\tau > 0$ and that $v(0) = 0$. Set moreover $T= v(\tau)$, $u_0 = u(0)$, $u_1 = u(\tau)$, $\nu_0 = \nu(0)$, and let $H_0^N$ be the constant value of $H^N$ along $(z^N,\zeta^N)$. Then, by Lemma \ref{lem:variablechange}, $\nu_0 = (2T)^{-1} (e^{2u_1}-e^{2u_0}) + H_0^N T$. Moreover, in the case $H_0^N \neq 0$,
\begin{equation}\label{eq:endtime}
\tau = v^{-1}(T) = \frac{1}{\omega} \left( \arctanh \frac{\nu_0}{\omega} + \arctanh \frac{2H_0^N T -\nu_0}{\omega} \right),
\end{equation}
where $\omega = \sqrt{2e^{u_0} H_0^N + \nu_0^2}$, whereas, in the case $H_0^N = 0$,
\[
\tau = \begin{cases}
u^{-1}(u_1) = \frac{u_1-u_0}{\nu_0} = \frac{2u_1-2u_0}{e^{2u_1}-e^{2u_0}} T  & \text{if $\nu_0\neq 0$,} \\
v^{-1}(T) = e^{-2u_0} T & \text{if $\nu_0=0$.}
\end{cases}
\]
Note that $L^N = T \sqrt{2H_0^N}$, whereas $L = \tau \sqrt{2e^{u_0} H_0^N + \nu_0^2}$. Easy manipulations of the above expressions then yield \eqref{eq:lengths}. For example, in the case $H_0^N > 0$, it is $L = \tau \omega$ and \eqref{eq:lengths} can be obtained by multiplying by $\omega$ both sides of \eqref{eq:endtime}, taking the $\cosh$ of both sides and applying the addition formula for $\cosh$.
\end{proof}

We can now turn the relation \eqref{eq:lengths} between lengths into a relation between sub-Riemannian distances.
We should mention that formula \eqref{eq:distanza} below was already given without any proof in \cite[p.~9]{HEB}. The argument given here can be thought of as a precise proof of it.

\begin{prop}\label{prp:distance}
For all $(z_0,u_0),(z_1,u_1) \in G$,
\begin{equation}\label{eq:distanza}\begin{split}
\dist((z_0,u_0),(z_1,u_1)) &= \arccosh \frac{1+e^{2(u_1-u_0)}+(e^{-u_0} \dist^N(z_0,z_1))^2}{2e^{u_1-u_0}} \\
&= \arccosh \left(\cosh(u_0-u_1) + e^{-(u_0+u_1)} \dist^N(z_0,z_1)^2/2 \right).
\end{split}\end{equation}
\end{prop}
\begin{proof}
By left-invariance of $\dist$ and $\dist^N$, it is sufficient to check the above formula in the case $(z_0,u_0) = 0_G$.

By the results in \cite{agrachev_subriemannian_2009} there exists an open dense subset $\Omega$ of $G$ made of points which are joined to the origin $0_G$ by a unique length-minimizing curve and this curve is a projection of an HJ-curve; analogously there exists an open dense subset $\Omega^N$ of $N$ made of points which are joined to the origin $0_N$ by a unique length-minimizing curve and this curve is the projection of an HJ-curve.

Let $\tilde\Omega = \Omega \cap (\Omega^N \times A)$. Then $\tilde\Omega$ is a dense open subset of $G$. Moreover, for all $(z_1,u_1) \in \Omega$, if $(z,u,\zeta,\nu)$ is the length-minimizing HJ-curve on $G$ joining $0_G$ to $(z_1,u_1)$, then the length $L$ of this curve coincides with $\dist(0_G,(z_1,u_1))$. Moreover, by Proposition \ref{prp:geodesic_G_N}, $(z,u,\zeta,\nu)$ is of the form $(\zeta^N \circ v,u,\zeta^N \circ v,\nu)$ for some HJ-curve $(z^N,\zeta^N)$ on $N$, whose length $L^N$ is related to $L$ by \eqref{eq:lengths}.

We now claim that $L^N = \dist^N(0_N,z_1)$. If not, the length-minimizing HJ-curve on $N$ joining $0_N$ to $z_1$ (which exists because $z_1 \in \Omega^N$) would have length less than $L^N$. So, via Proposition \ref{prp:geodesic_N_G}, we could construct an HJ-curve on $G$ joining $0_G$ to $(z_1,u_1)$ with length less than $L$, which would lead to a contradiction.

The relation \eqref{eq:lengths} between lengths yields \eqref{eq:distanza} for all $(z_1,u_1) \in \tilde \Omega$. Since $\tilde\Omega$ is dense and $\dist,\dist^N$ are continuous, \eqref{eq:distanza} holds for all $(z_1,u_1) \in G$.
\end{proof}

\subsection{Volume asymptotics and integral formulas for radial functions}\label{subs:integralradial}
The expression \eqref{eq:distanza} for the sub-Riemannian distance $\dist$ allows us to give precise formulas and asymptotics for the volume of the corresponding balls. It should be noted that detailed information on the local behavior of $\dist$ could be deduced by the ball-box theorem (see \cite{nagel_balls_1985} or \cite{Montgomery}). For the global behavior, however, sufficiently precise general results seem not to be available and formula \eqref{eq:distanza} becomes crucial.

We shall obtain the volume formulas as corollaries of integral formulas for radial functions. By radial function on $G$ we mean a function of the form $x \mapsto f(|x|_\dist)$, where $f : \Rnon \to \CC$ and $|x|_\dist = \dist(x,0_G)$ is the distance of $x \in G$ from the origin. Analogously by radial function on $N$ we mean a function of the form $z \mapsto f(|z|_N)$, where $|z|_N  = \dist^N(z,0_N)$ is the distance of $z \in N$ from the origin.

The homogeneity of $\dist^N$ yields immediately the following integral formula for radial functions on $N$: for all Borel functions $f : \Rnon \to \Rnon$,
\begin{equation}\label{eq:radial_N}
\int_N f(|z|_N) \di z = V_N Q \int_0^\infty f(s) \, s^{Q-1} \di s,
\end{equation}
where $V_N = |B_N(0_N,1)|$. Clearly such a formula can be extended to complex-valued functions $f$, as soon as the integrals make sense. We now obtain a similar formula on $G$.

\begin{prop}\label{prp:radialdensity}
For all Borel functions $f : \Rnon \to \Rnon$,
\begin{equation}\label{eq:radialdensity}
\int_G f(|x|_\dist) \di\mu(x) = \int_G f(|x|_\dist) \,m(x)\di\mu(x)
= c_N Q \int_{0}^\infty f(r) \, \sinh^Q r \di r,
\end{equation}
where $c_N = V_N 2^{Q-1} \frac{\Gamma(Q/2)^2}{\Gamma(Q)}$. In particular
\begin{equation}\label{eq:volumeasymptotics}
\mu\big(B_\dist(0,r)\big) = c_N Q \int_0^r \sinh^Q s \di s \sim \begin{cases}
r^{Q+1} &\text{if $0 < r \leq 1$,}\\
e^{Qr} &\text{if $r \geq 1$.}\end{cases}
\end{equation}
\end{prop}
\begin{proof}
Since $|x|_\dist=|x^{-1}|_\dist$ by left-invariance of $\dist$ (cf.\ \cite[\S III.4, p.\ 40]{varopoulos_analysis_1992}),
\[
\int_G f(|x|_\dist) \di \mu(x) =
\int_G f(|x^{-1}|_\dist) \di \mu(x) 
=\int_G f(|x|_\dist) \,m(x)\di \mu(x).
\]
Moreover, by formulas \eqref{eq:distanza} and \eqref{eq:radial_N},
\[\begin{split}
&\int_G f(|x|_\dist) \di \mu(x) \\
&= V_N Q\int_{-\infty}^\infty \int_0^\infty f(\arccosh(\cosh u + e^{-u} s^2/2)) \,s^{Q-1}\di s \di u \\
&= V_N Q \, 2^{Q/2-1} \int_{-\infty}^\infty \int_0^\infty f(\arccosh(\cosh u + s)) \, e^{Qu/2} \,s^{Q/2-1}\di s \di u \\
&= V_N Q \, 2^{Q/2-1} \int_{0}^\infty f(r) \, \sinh r \int_{-r}^r e^{Qu/2} \,(\cosh r - \cosh u)^{Q/2-1} \di u \di r
\end{split}\]
(in the last step the change of variable $s = \cosh r - \cosh u$ was used). One can explicitly evaluate the inner integral in the last formula and obtain
\[
\int_{-r}^r e^{Qu/2} \,(\cosh r - \cosh u)^{Q/2-1} \di u =2^{Q/2} \frac{\Gamma(Q/2)^2}{\Gamma(Q)} \sinh^{Q-1} r.
\]
This gives \eqref{eq:radialdensity} and \eqref{eq:volumeasymptotics} follows by taking $f = \chi_{[0,r)}$.
\end{proof}

Similar computations give us expressions for weighted integrals of radial functions, that will be useful in the sequel. Define the weight $w$ on $G$ by $w(z,u) = |z|_N^Q$. Then the following result holds.

\begin{prop}
There exists a constant $C_Q > 0$ such that for all Borel functions $f : \Rnon \to \Rnon$,
\begin{equation}\label{eq:weightedradialcomparison}
\int_G m(x) f(|x|_\dist) \, w(x) \di \mu(x) \leq C_Q \int_G f(|x|_\dist) \, |x|_\dist \di \mu  (x).
\end{equation}
Moreover
\begin{equation}\label{eq:inverseweight}
\int_{B_\dist(0,r)} (1+w)^{-1} \di \mu \lesssim \begin{cases}
r^{Q+1} &\text{if $0 < r \leq 1$,} \\
r^{2}   &\text{if $r \geq 1$.}
\end{cases}
\end{equation}
\end{prop}
\begin{proof}
A simple modification of the proof of Proposition \ref{prp:radialdensity} gives the following integral formula:
\begin{multline}\label{eq:integralweighted}
\int_G m(x) f(|x|_\dist) \, w(x) \di \mu(x) \\
= 2^{Q-1} V_N Q \int_{0}^\infty f(r) \, \sinh {r} \int_{-r}^r (\cosh r - \cosh u)^{Q-1} \di u  \di r.
\end{multline}
Since $\int_{-r}^r (\cosh r - \cosh u)^{Q-1} \di u \lesssim r \sinh^{Q-1} r$, the estimate \eqref{eq:weightedradialcomparison} follows by comparison of \eqref{eq:radialdensity} and \eqref{eq:integralweighted}.

As for \eqref{eq:inverseweight}, this is clear by \eqref{eq:volumeasymptotics} in the case $r \leq 1$. If instead $r \geq 1$, then
\[\begin{split}
&\int_{B_\dist(0,r)} (1+w)^{-1} \di \mu \\
&= V_N Q\int_{-\infty}^\infty \int_0^\infty \chi_{[0,r)}\big(\arccosh(\cosh u + e^{-u} s^2/2)\big) \,\frac{s^{Q-1}}{1+s^Q} \di s \di u \\
&\sim \int_{-r}^{r} \int_0^{2e^u (\cosh r - \cosh u)} \,\frac{s^{Q/2-1}}{1+s^{Q/2}}\di s \di u\\
&\lesssim \int_{-r}^{r} \int_0^{2e^{2r}} \,\frac{1}{1+s}\di s \di u \sim r^2
\end{split}\]
and we are done.
\end{proof}

\section{\CZ\ theory}\label{s:CZ}

\subsection{Abstract \CZ\ theory}\label{subs:abstractCZ}

It is well known that in spaces of homogeneous type integrable functions admit a \CZ\ decomposition and that in this context the classical \CZ\ theory for singular integrals and the theory of Hardy and $BMO$ spaces \cite{S} can be generalized \cite{CW,CW2}. However, because of exponential volume growth, the group $G$ under consideration is not a space of homogeneous type and a further generalization of the \CZ\ theory is necessary. This generalization was introduced by Hebisch and Steger \cite{HS} and further developed by Vallarino \cite{V1}. Here we summarize some of the results of this theory that will be used in the sequel.

\begin{defi}\label{dfn:CZdef}
A \emph{CZ-space} is a metric measure space $(X,d,\mu)$ such that there exist a positive constant $\kappa_0$ and a family $\mathcal{R}$ of measurable subsets of $X$ with the following properties: for all $R \in \mathcal{R}$̧, there exist $x \in X$ and $r > 0$ such that
\begin{enumerate}[label=(\roman*)]
\item\label{en:CZadm1} $R \subset B(x,\kappa_0 r)$,
\item\label{en:CZadm2} $\mu(R^*) \leq \kappa_0 \mu(R)$, where $R^* = \{x \in X \tc d(x,R) < r\}$;
\end{enumerate}
moreover, for all $f \in L^1(X)$ and for all $\alpha>\kappa _0~\frac{\|f\|_1}{\mu(X)}~$ ($\alpha >0~{\rm if}~\mu(X)=\infty$) there exists a decomposition $f=g+\sum_{i\in \NN}b_i$ and sets $R_i \in \mathcal{R}$ such that
\begin{enumerate}[label=(\roman*),resume]
\item\label{en:CZdec1} $\|g\|_\infty \leq \kappa_0 \alpha$,
\item\label{en:CZdec2} $\supp b_i \subset R_i$ and $\int b_i \di\mu=0$ for all $i\in \NN$,
\item\label{en:CZdec3} $\sum_i \mu(R_i)\leq \kappa_0\,\frac{\|f\|_1}{\alpha}$,
\item\label{en:CZdec4} $\sum_i {\|b_i\|_1}\leq \kappa_0\,\|f\|_{1}$.
\end{enumerate}
\end{defi}

The constant $\kappa_0$ is called the \emph{CZ-constant} of $(X,d,\mu)$. A decomposition $f=g+\sum_{i\in \NN}b_i$ which has properties \ref{en:CZdec1}-\ref{en:CZdec4} of Definition \ref{dfn:CZdef} is said to be a \emph{\CZ\ decomposition} of $f$ at height $\alpha$. The elements of the family $\mathcal{R}$ are called \emph{admissible sets} and, for each $R \in \mathcal{R}$, the point $x \in X$ and the number $r > 0$ satisfying properties \ref{en:CZadm1}-\ref{en:CZadm2} of Definition \ref{dfn:CZdef} are called the \emph{center} and the \emph{radius} of $R$ respectively.

Note that the above definition of CZ-space is more restrictive than the definition of ``\CZ\ space'' given by Hebisch and Steger in \cite{HS}. Hence the following boundedness theorem for a class of linear operators on CZ-spaces is a consequence of \cite[Theorem 2.1]{HS}.

\begin{teo}\label{thm:Teolim}
Let $(X,d,\mu)$ be a CZ-space.
Let $T$ be a linear operator bounded on $L^2(X)$ such that $T=\sum_{j\in \ZZ}T_j$, where
\begin{enumerate}[label=(\roman*)]
\item the series converges in the strong topology of operators on $L^2(X)$;
\item every $T_j$ is an integral operator with kernel $K_j$;
\item there exist positive constants $b,B,\varepsilon$ and $c>1$ such that
\begin{gather*}
\int_X|K_j(x,y)|\,\big(1+c^j d(x,y)\big)^{\varepsilon}\,{\rm{d}}\mu (x) \leq B\qquad\forall y\in X,\\
\int_X|K_j(x,y)-K_j(x,z)| \,{\rm{d}}\mu (x) \leq B\,\big(c^j d(y,z)\big)^b\qquad\forall y,z\in X.
\end{gather*}
\end{enumerate}
Then $T$ extends from $L^1(X)\cap L^2(X)$ to an operator of weak type $(1,1)$ and bounded on $L^p(X)$, for $1<p\leq 2$.
\end{teo}

In \cite{V1} it was noticed that if a CZ-space satisfies an additional condition, then one can develop an $H^1$-$BMO$ theory on it.

\begin{defi}
We say that the CZ-space $(X,d,\mu)$ with family of admissible sets $\mathcal{R}$ satisfies condition \condC\ if there exists a subfamily $\mathcal{R}'$ of $\mathcal{R}$ with the following properties:
\begin{enumerate}[label=(\roman*),resume]
\item given $R_1,\,R_2$ in $\mathcal R'$ such that $R_2\cap  R_1\neq \emptyset$, then either $R_1\subset R_2$ or $R_2\subset R_1$;
\item for every set $R$ in $\mathcal R$ there exists a set $R'$ in $\mathcal R'$ which contains $R$.
\end{enumerate}
\end{defi}

Suppose now that $(X,d,\mu)$ is a CZ-space with family of admissible sets $\mathcal{R}$ which satisfies condition \condC. Then we introduce an atomic Hardy space $H^{1}$ and a space of bounded mean oscillation functions on $X$ as follows.

\begin{defi}\label{dfn:atom}
An \emph{atom} is a function $a$ in $L^1(X)$ such that
\begin{enumerate}[label=(\roman*)]
\item $a$ is supported in an admissible set $R \in \mathcal{R}$;
\item $\|a\|_{2}\leq \mu({R})^{-1/2}\,;$
\item $\int_S a\di\mu =0$\,.
\end{enumerate}
\end{defi}

\begin{defi}\label{dfn:Hardy}
The Hardy space $H^{1}(X)$ is the space of all functions $f$ in $L^1(X)$ such that $f=\sum_j \lambda_j\, a_j$, where the $a_j$ are atoms and $\lambda _j$ are complex numbers such that $\sum _j |\lambda _j|<\infty$. It is a Banach space with the norm
\[
\|f\|_{H^{1}} = \inf \left\{ \sum_j|\lambda_j| \tc f=\sum_j \lambda_j \,a_j, \,\, a_j \text{ atoms}, \, \lambda_j \in \CC\right\}.
\]
By $H^{1}_{\rm{fin}}(X)$ we denote the subspace of $H^{1}(X)$ of finite linear combinations of atoms.
\end{defi}

\begin{defi}\label{dfn:BMO}
The space $\mathcal{B}\mathcal{M}\mathcal{O}(X)$ is the space of all functions $f$ in $L^2_{\loc}(X)$ such that
\[
\sup_{R\in \mathcal{R}} \left(\frac{1}{\mu(R)}\int_R|f-f_R|^2\di\mu \right)^{1/2}<\infty,
\]
where $f_R=\frac{1}{\mu(R)}\int_Rf\di \mu$. The space $BMO(X)$ is the quotient of $\mathcal{B}\mathcal{M}\mathcal{O}(X)$ modulo constant functions. It is a Banach space with the norm
\[
\|f\|_{BMO}=\sup\left\{\left(\frac{1}{\mu(R)}\int_R|f-f_R|^2\di\mu \right)^{1/2}\,:~R\in\mathcal R \right\}.
\]
\end{defi}

For more details on the spaces $H^1(X)$ and $BMO(X)$ we refer the reader to \cite{V1}. In particular, the space $BMO(X)$ can be identified with the the dual of $H^{1}(X)$ \cite[Theorem 3.9]{V1}.

\begin{prop}\label{prp:H1BMOduality}
The following hold.
\begin{enumerate}[label=(\roman*)]
\item For any $g$ in $BMO(X)$ the functional $\Lambda$ defined on $H^{1}_{\rm{fin}}(X)$ by
\[
\Lambda(f)=\int f\,g\di\mu\qquad \forall f\in H^{1}_{\rm{fin}}(X) ,
\]
extends to a bounded functional on $H^{1}(X)$. Furthermore, there exists a constant $C$ such that $\|\Lambda\|_{(H^{1}(X))^*}\leq C\,\|f\|_{BMO}$.
\item For any bounded linear functional $\Lambda$ on $H^{1}(X)$ there exists a function $g$ in $BMO(X)$ such that $\Lambda(f)=\int g\,f\di\mu$ for all $f$ in $ H^{1}_{\rm{fin}}(X)$ and $\|g \|_{BMO}\leq C\,\|\Lambda\|_{(H^{1}(X))^*}$, with $C$ independent of $g$ and $\Lambda$.
\end{enumerate}
\end{prop}

Moreover, the following $H^1$-$L^1$ boundedness result holds for singular integral operators on CZ-spaces \cite[Theorem 3.10]{V1}.
\begin{teo}\label{thm:TeolimH1}
Let $(X,d,\mu)$ be a CZ-space which satisfies condition \condC. Let $T$ be a linear operator which satisfies the hypotheses of Theorem \ref{thm:Teolim}. Then $T$ is bounded from $H^{1}(X)$ to $L^1(X)$.
\end{teo}

\subsection{\CZ\ theory on \texorpdfstring{$(G,\varrho,\mu)$}{G}}\label{subs:CZG}

We shall prove that the space $(G,\varrho,\mu)$ is a CZ-space in the sense defined in the previous subsection. This fact was already announced and proved by Hebisch in \cite{HEB2} in a more general class of amenable Lie groups, including the groups we are considering here. However, for these groups the construction of the \CZ\ decomposition becomes more transparent than the one given in \cite{HEB2} and we think that it is worthwhile to see the explicit construction in our setting. Moreover this construction allows us to show that the CZ-space $(G,\varrho,\mu)$ satisfies condition \condC\ and consequently a theory of Hardy spaces can be developed on $G$.

The difficulty in the construction consists in the definition of a suitable family $\mathcal{R}$ of admissible sets on $G$. We cannot use balls as in the classical case, because their measure increases exponentially and condition \ref{en:CZadm2} of Definition \ref{dfn:CZdef} would not be satisfied. To define admissible sets we adapt to the sub-Riemannian distance the ideas of \cite{HS} and \cite{V2}.

Christ \cite[Theorem 11]{C} proved the existence of a family of dyadic sets in a space of homogeneous type, which can be formulated for the stratified group $N$ as follows.
\begin{teo}\label{thm:dyadic}
There exist constants $\rd,C_N>1$, an integer $J \geq 2$ and a collection of Borel subsets $Q_{\alpha}^k \subset N$ and points $n_\alpha^k \in N$, where $k \in \ZZ$, $\alpha \in I_k$ and $I_k$ is a countable index set,
such that, for all $k \in \ZZ$, the following hold:
\begin{enumerate}[label=(\roman*)]
\item\label{en:dyadic1} $|N-\bigcup_{\alpha\in I_k}Q_{\alpha}^k|=0$;
\item\label{en:dyadic2} $B_N(n_{\alpha}^k,C_N^{-1}\,\rd^k)\subset Q_{\alpha}^k\subset B_N(n_{\alpha}^k,C_N\,\rd^k)$ for all $\alpha \in I_k$;
\item\label{en:dyadic3} $Q_{\alpha}^k \cap Q_{\beta}^k =\emptyset$ for all $\alpha,\beta \in I_k$ with $\alpha \neq \beta$;
\item\label{en:dyadic4} for all $\alpha \in I_k$, $Q_{\alpha}^k$ has at most $J$ subsets of the form $Q_{\beta}^{k-1}$ for $\beta \in I_{k-1}$;
\item\label{en:dyadic5} for all $\ell\leq k$ and $\beta \in I_{\ell}$ there is a unique $\alpha \in I_k$ such that $Q_{\beta}^{\ell}\subset Q_{\alpha}^k $;
\item\label{en:dyadic6} for all $\ell\leq k$, $\alpha \in I_k$ and $\beta \in I_\ell$, either $Q_{\alpha}^k\cap Q_{\beta}^{\ell}=\emptyset$ or $Q_{\beta}^{\ell}\subset Q_{\alpha}^k $.
\end{enumerate}
\end{teo}

Let us fix a system of dyadic sets $Q_{\alpha}^k$, points $n_{\alpha}^k$, index sets $I_k$ and constants $\rd,C_N,J$ as in Theorem \ref{thm:dyadic}.
Further let us fix two positive constants $M$ and $r_0$
such that the following conditions are satisfied:
\begin{align}
\label{eq:Mr0_A} 1 &< r_0<2\log 2 \\
\label{eq:Mr0_B} M &>1 \\
\label{eq:Mr0_C} e^{r_0}e^{2M}r_0 &\leq e^{2Mr_0} \\
\label{eq:Mr0_D} 6M &>\log\rd-\log 2+\frac{r_0}{2}\\
\label{eq:Mr0_E} \rd e^{4Mr_0} &< 2e^{8M} \inf\{ re^{-r/2}:\, r_0<r\leq 2r_0\}  \\
\label{eq:Mr0_F} \rd &< 4e^{(4M-1)r_0}\,.
\end{align}
We define admissible sets as the product of dyadic sets in $N$ and intervals in $A$ as follows.
\begin{defi}\label{dfn:admissibility}
An \emph{admissible set} in $G$ is a set of the form
\[
Q_{\alpha}^k\times (u_0-r,u_0+r),
\]
where $k \in \ZZ$, $\alpha \in I_k$, $u_0\in\RR$, $r>0$ are such that
\begin{equation}\label{eq:adm}
\begin{aligned}
r \, e^{2M}\,e^{u_0}&\leq \rd^k < 4 \, r \, e^{8M} \,e^{u_0} &&\text{if } 0<r\leq r_0,\\
e^{2Mr}\,e^{u_0} &\leq \rd^k < 4 \, e^{8Mr} \,e^{u_0} &&\text{if } r>r_0.
\end{aligned}
\end{equation}
We shall call \emph{small admissible set} an admissible set corresponding to a parameter $r\in (0,r_0]$ and \emph{big admissible set} an admissible set corresponding to a parameter $r\in(r_0,\infty)$.
We denote by $\mathcal R$ the family of all admissible sets in $G$.
\end{defi}

Proposition \ref{prp:distance} allows us to obtain precise relations between balls and ``rectangles'' on $G$, which will be important in the following.

\begin{prop}\label{prp:balls}
There exists a positive constant $C_1$ such that
\begin{enumerate}[label=(\roman*)]
\item\label{en:balls1} $B_N\big(0_N, 4C_N  e^{8M} \,  r  \big)\times (-r, r)\subset B_{\varrho}(0_G,C_1 r)$ for every $r\in (0,\infty)$;
\item\label{en:balls2} $B_N\big(0_N, 4C_N  e^{8Mr}  \big)\times (-r, r)\subset B_{\varrho}(0_G,C_1 r)$ for every $r\in (r_0,\infty)$;
\item\label{en:balls3} $B_{\varrho}(0_G,r)\subset B_N\big(0_N, e^r  \big)\times (-r,r)$ for every $r\in (0,\infty)$;
\item\label{en:balls4} $B_{\varrho}(0_G,r)\subset B_N\big(0_N, C_1 r  \big)\times (-r,r)$ for every $r\in (0,r_0]$.
\end{enumerate}
\end{prop}
\begin{proof}
We first prove \ref{en:balls1}. If $(z,u)\in B_N\big(0_N, 4C_N  e^{8M} \,  r  \big)\times (-r, r)$, then, by formula \eqref{eq:distanza},
\[\begin{split}
\varrho\big((z,u),0_G\big)&< \arccosh \left(\cosh r+\frac{16e^{r}C_N^2e^{16M}r^2   }{2}\right)\\
&\leq \arccosh \cosh (C_1r),
\end{split}\]
for a sufficiently large $C_1$ and for every $r\in (0,\infty)$.

We now prove \ref{en:balls2}. If $(z,u)\in B_N\big(0_N, 4C_N  e^{8Mr} \big)\times (-r, r)$, then, by formula \eqref{eq:distanza},
\[
\varrho\big((z,u),0_G\big) < \arccosh \left(\cosh r+\frac{16e^{r}C_N^2e^{16Mr}}{2}\right)\leq \arccosh \cosh (C_1r),
\]
for a sufficiently large $C_1$ and for every $r\in (r_0,\infty)$.

We now consider any point $(z,u)\in B_{\varrho}(0_G,r)$. By formula \eqref{eq:distanza} it is obvious that $\cosh u < \cosh r$ and then $u\in (-r,r)$. Suppose now that $|z|\geq e^r$. Then
\[
\varrho\big((z,u),0_G\big) \geq \arccosh \left(1+\frac{e^{-r}e^{2r}   }{2}\right)\geq \arccosh \cosh r=r.
\]
Then $|z| < e^r$ and \ref{en:balls3} is proved. Take now any point $(z,u)\in B_{\varrho}(0_G,r)$ and suppose that $|z|\geq C_1\,r$. Then
\[
\varrho\big((z,u),0_G\big) \geq \arccosh \left(1+\frac{e^{-r}C_1^2r^{2}   }{2}\right)\geq \arccosh \cosh r=r,
\]
for every $r\in (0,r_0]$, if $C_1$ is chosen sufficiently large. Then $|z| < C_1r$ and \ref{en:balls4} is proved.
\end{proof}

We now investigate some properties of admissible sets.

\begin{prop}\label{prp:proprieta}
There exists a positive constant $C^*$ such that, for every admissible set $R=Q_{\alpha}^k\times (u_0-r,u_0+r)$, the following hold:
\begin{enumerate}[label=(\roman*)]
\item\label{en:proprieta1} $R\subset B_{\varrho}\big((n_{\alpha}^k, u_0),C_1\, r\big)$, where $C_1$ is the constant which appears in Proposition \ref{prp:balls};
\item\label{en:proprieta2} $\mu\big(R^*\big)\leq C^*\mu\big(R\big)\,,$ where $R^*=\{(z,u)\in G: \varrho\big((z,u),R\big)< r\}$.
\end{enumerate}
\end{prop}
\begin{proof}
\emph{Case $0<r\leq r_0$.} By Theorem \ref{thm:dyadic} and Definition \ref{dfn:admissibility},
\[\begin{split}
R &\subset B_N\big(n^k_{\alpha},4C_N   \, e^{8M} \,e^{u_0}\, r    \big)\times (u_0-r,u_0+r) \\
&=(n^k_{\alpha},u_0)\cdot B_N\big(0_N,4C_N  e^{8M} \,  r   \big)\times (-r, r)\,.
\end{split}\]
By Proposition \ref{prp:balls} $B_N\big(0_N, 4C_N  e^{8M} \,  r  \big)\times (-r, r)\subset B_{\varrho}(0_G,C_1 r)$, which implies \ref{en:proprieta1}.

To prove \ref{en:proprieta2} we remark that $R^*=\bigcup_{(z,u)\in R} B_{\varrho}\big((z,u),r\big)\,.$
By the left-invariance of the metric and Proposition \ref{prp:balls}, for every $(z,u)\in R$,
\[\begin{split}
B_{\varrho}\big((z,u), r\big)&=(z,u)\cdot B_{\varrho}(0_G, r)\\
&\subset (z,u)\cdot B_N\big(0_N,C_1\, r\big)\times (-r,r)\\
&=B_N\big(z,C_1\,e^{u} r\big)\times (u-r,u+r)\\
&\subset B_N\big(n^k_{\alpha},C_1\,e^{u} r +C_N \rd^k  \big)\times (u_0-2r, u_0+2r)\\
&\subset B_N\big(n^k_{\alpha},C\,e^{u_0}  r    \big)\times ( u_0-2r, u_0+2r),
\end{split}\]
where $C = C_1 e^{r_0} + 4C_N e^{8M}$ and we have applied the triangle inequality in $N$ and the admissibility condition. This implies that
\[
\mu(R^*)\lesssim e^{Qu_0} r^Q \,r\sim \rd^{Qk} \,r\sim \mu( R),
\]
which gives \ref{en:proprieta2}.

\bigskip

{\emph{Case $r> r_0$.}} To prove \ref{en:proprieta1} note that by Theorem \ref{thm:dyadic}
\[
R \subset B_N\big(n_{\alpha}^k,C_N\,\rd^k\big)\times (u_0-r,u_0+r)\,,
\]
which is contained in $B_N\big(n_{\alpha}^k, 4C_N e^{8Mr} e^{u_0})\times  (u_0-r,u_0+r)$ by the admissibility condition \eqref{eq:adm}. By the left-invariance of the metric  and Proposition \ref{prp:balls}
\[\begin{split}
R &\subset (n_{\alpha}^k,u_0)\cdot B_N\big(0_N,4C_N\,  e^{8Mr})\times (-r,r) \\
&\subset (n_{\alpha}^k, u_0)\cdot B_{\varrho}(0_G,C_1 \, r)\\
&=B_{\varrho}\big((n_{\alpha}^k, u_0), C_{1}\, r \big).
\end{split}\]
To prove \ref{en:proprieta2} we remark that $R^*=\bigcup_{(z,u)\in R} B_{\varrho}\big((z,u),  r\big)$. By the left-invariance of the metric and Proposition \ref{prp:balls} for every $(z,u)\in R$
\[\begin{split}
B_{\varrho}\big((z,u),  r\big)&=(z,u)\cdot B_{\varrho}(0_G, r)\\
&\subset (z,u)\cdot B_{\varrho}\big(0_N, e^r \big)\times (-r,r)\\
&=B_N\big( z,    e^ue^r\big)\times (u-r,u+r).
\end{split}\]
Using the fact that $(z,u)\in R$ and the admissibility condition on $R$, we see that
\[
(u-r,u+r) \subset(u_0-2r,u_0+2r)
\]
and
\[\begin{split}
B_N\big(z, e^{u}e^r\big)& \subset B_N\big(z,\,e^{u_0+r}e^r\big)\\
&\subset B_N\big(n_{\alpha}^k, e^{u_0+2r}+C_N\,\rd^k\big)\\
&\subset B_N\big(n_{\alpha}^k,(1+C_N)\,\rd^k\big) .
\end{split}\]
Thus
\[
R^*\subset B_N\big(n_{\alpha}^k,(1+C_N)\rd^k\big)\times (u_0-2r,u_0+2r),
\]
and so
\[
\mu\big(R^*\big) \lesssim |B_N\big(n_{\alpha}^k, \,\rd^k\big)|\, r\sim \mu\big(R\big),
\]
as required.
\end{proof}

We now define a way of splitting an admissible set into at most $J$ disjoint admissible subsets, where $J$ is the constant which appears in Theorem \ref{thm:dyadic}.

\begin{defi}
An admissible set $R=Q_{\alpha}^k\times (u_0-r,u_0+r)$ is called \emph{strongly admissible} if \eqref{eq:adm} also holds with $k-1$ in place of $k$, that is, if
\[\begin{aligned}
r \, e^{2M}\,e^{u_0}&\leq \rd^{k-1} < 4 \, r \, e^{8M} \,e^{u_0} &&\text{when } 0<r\leq r_0,\\
e^{2Mr}\,e^{u_0} &\leq \rd^{k-1} < 4 \, e^{8Mr} \,e^{u_0} &&\text{when } r>r_0.
\end{aligned}\]
\end{defi}

Note that the upper bound for $\eta^{k-1}$ in the above inequalities is automatically satisfied because $R$ is admissible and $\eta^{k-1} < \eta^k$; the additional requirement for $R$ to be strongly admissible is the lower bound for $\eta^{k-1}$.

\begin{defi}\label{dfn:children}
For all admissible sets $R=Q_{\alpha}^k\times (u_0-r,u_0+r)$, we define the \emph{children} of $R$ as follows: if $R$ is strongly admissible, then the children of $R$ are all the sets of the form
\[
Q_\beta^{k-1}\times (u_0-r,u_0+r)
\]
where $\beta \in I_{k-1}$ and $Q_\beta^{k-1} \subset Q_\alpha^k$; otherwise the children of $R$ are the sets
\[
Q_{\alpha}^k\times (u_0-r,u_0) \qquad\text{and}\qquad Q_{\alpha}^k\times (u_0,u_0+r).
\]
We denote by $\chil(R)$ the set of the children of $R$.
\end{defi}

\begin{defi}
Let $E$ be a measurable subset of a measure space. A \emph{quasi-partition} of $E$ is an at most countable family of non-negligible, pairwise disjoint measurable subsets of $E$, whose union has full measure in $E$.
\end{defi}

\begin{lem}\label{lem:tagli}
Let $C_2 = \max \{2, (C_N^{2} \rd)^Q\}$. Then, for all admissible sets $R$, the following hold:
\begin{enumerate}[label=(\roman*)]
\item\label{en:tagli1} $R$ has at most $J$ children.
\item\label{en:tagli2} $\chil(R)$ is a quasi-partition of $R$.
\item\label{en:tagli3} $C_2^{-1} \mu( R ) \leq \mu(R') \leq \mu(R)$ for all $R' \in \chil(R)$.
\item\label{en:tagli4} All the children of $R$ are admissible.
\end{enumerate}
\end{lem}
\begin{proof}
Let $R = Q_\alpha^k \times (u_0-r,u_0+r)$. Since $R$ is admissible, \eqref{eq:adm} holds. Suppose that $R$ is strongly admissible. Then the children of $R$, that is, the sets of the form
\[
Q_\beta^{k-1}\times (u_0-r,u_0+r),
\]
where $\beta \in I_{k-1}$ and $Q_\beta^{k-1} \subset Q_\alpha^k$, are admissible too. Moreover, from the properties of dyadic sets given by Theorem \ref{thm:dyadic} it is clear that properties \ref{en:tagli1},\ref{en:tagli2},\ref{en:tagli3} hold.

Suppose instead that $R$ is not strongly admissible. Then, when $r\leq r_0$, it must be
\begin{equation}\label{eq:noadmisrpic}
r \, e^{2M} e^{u_0} \leq \rd^k <\rd\,  r \, e^{2M} e^{u_0},
\end{equation}
while, when $r>r_0$,
\begin{equation}\label{eq:noadmisrgr}
e^{2Mr} e^{u_0} \leq \rd^k < \rd\, e^{2Mr} e^{u_0}.
\end{equation}
Moreover the children of $R$ are the sets
\[
R_1 = Q_{\alpha}^k\times (u_0-r,u_0) \qquad\text{and}\qquad R_2 = Q_{\alpha}^k\times (u_0,u_0+r),
\]
which are ``centered'' at $(n_{\alpha}^k,{u_0}-r/2)$ and $(n_{\alpha}^k,u_0+ r/2)$ respectively, and it is clear that properties \ref{en:tagli1},\ref{en:tagli2},\ref{en:tagli3} hold.
We shall prove that $R_1$ and $R_2$ are admissible: to do so, we distinguish three cases.

\emph{Case $r\leq r_0$.} In this case $R$ is a small admissible set and we must prove that $R_1,R_2$ are both small admissible sets, because $r/2\leq r_0$. Notice that
\[\begin{split}
e^{u_0-r/2}e^{2M}\frac {r}{2}&\leq e^{u_0+r/2}e^{2M}\frac{r}{2}\\
&\leq \frac12 \rd^k\,e^{r/2}\\
&\leq \frac12 \rd^k\,e^{\frac{r_0}{2}} \\
&\leq \rd^k,
\end{split}\]
since $r_0\leq 2\log 2$ by \eqref{eq:Mr0_A}. Moreover,
\[\begin{split}
\rd^k   &<\rd\,  e^{u_0}e^{2M}\,r\\
&<   \,4e^{8M}e^{u_0-r/2}\,\frac{r}{2}\\
&<4e^{8M}e^{u_0+r/2}\,\frac{r}{2},
\end{split}\]
since $\rd<2e^{6M}e^{-\frac{r_0}{2}}$ by condition \eqref{eq:Mr0_D}. This proves that $R_1$ and $R_2$ are admissible in this case.

\emph{Case $r_0<r\leq 2r_0$.} In this case $R$ is a big admissible set and we must prove that $R_1,R_2$ are both small admissible sets, because $r/2\leq r_0$. Notice that
\[\begin{split}
e^{u_0-r/2}e^{2M}\frac{r}{2}&\leq e^{u_0+r/2}e^{2M} \frac{r}{2}\\
&\leq  e^{u_0 }e^{2Mr}\\
&\leq \rd^k,
\end{split}\]
since $e^{r_0}e^{2M}r_0\leq e^{2Mr_0}$ by condition \eqref{eq:Mr0_C}. Moreover,
\[\begin{split}
\rd^k   &<\rd\,  e^{u_0}e^{2Mr}\\
&< 4e^{8M}e^{u_0-r/2}r/2\\
&<4e^{8M}e^{u_0+r/2}r/2\,,
\end{split}\]
since $\rd e^{4Mr_0}<2e^{8M} \inf_{r_0<r\leq 2r_0} re^{-r/2}$ by condition \eqref{eq:Mr0_E}. This proves that $R_1$ and $R_2$ are admissible in this case.

\emph{Case $r>2r_0$.} In this case $R$ is a big admissible set and we must prove that $R_1$, $R_2$ are both big admissible sets because $\frac{r}{2}> r_0$. Notice that
\[\begin{split}
e^{u_0-r/2}e^{2Mr/2}&\leq e^{u_0+r/2}e^{2Mr/2}\\
&\leq  e^{u_0 }e^{2Mr}\\
&\leq \rd^k,
\end{split}\]
since $M>\frac12$ by \eqref{eq:Mr0_B}. Moreover,
\[\begin{split}
\rd^k   &<\rd\,  e^{u_0}e^{2Mr}\\
&< 4e^{8Mr/2}e^{u_0-r/2}\\
&<4e^{8Mr/2}e^{u_0+r/2},
\end{split}\]
since $\rd<4e^{(4M-1)r_0}$ by condition \eqref{eq:Mr0_F}. This proves that $R_1$ and $R_2$ are admissible also in this case.
\end{proof}

By adapting the proof of \cite[Lemma 3.16]{V1}, we can construct a quasi-partition of $G$ in big admissible sets whose measure is as large as we want.

\begin{lem}\label{lem:partizionegrande}
For all $\sigma>0$ there exists a quasi-partition $\mathcal P$ of $G$ in big admissible sets whose measure is greater than $\sigma$.
\end{lem}
\begin{proof}
Choose $r_1>r_0$ and $k_1\in \mathbb Z$ such that $e^{2Mr_1}\leq \rd^{k_1} <4e^{8Mr_1}$. Then the sets $R^1_{\alpha}=Q^{k_1}_{\alpha}\times  (-r_1,r_1)$, $\alpha\in I_{k_1}$,  are a quasi-partition of $N\times (-r_1,r_1)$ made of big admissible sets. It is possible to choose $k_1$ and $r_1$ in such a way that $|B_N(0_N,C_N^{-1} \rd^{k_1})| 2r_1 >\sigma$, so that $\mu(R^1_{\alpha})>\sigma$ for all $\alpha\in I_{k_1}$.

Suppose that a quasi-partition of $N\times (r_1+\dots +2r_{n-1},r_1+\dots+ 2r_{n-1}+2r_n)$ made of big admissible sets of measure greater than $\sigma$ has been constructed. Choose $r_{n+1}>r_0$ and $k_{n+1}\in \mathbb Z$ such that $e^{2Mr_{n+1}} e^{u_{n+1}}  \leq \rd^{k_{n+1}} <4e^{8Mr_{n+1}}e^{u_{n+1}}$, where $u_{n+1}=r_1+\dots+ 2r_{n}+r_{n+1}$. Then the sets $R^{n+1}_{\alpha}=Q^{k_{n+1}}_{\alpha}\times  ( u_{n+1}-r_{n+1}, u_{n+1}+r_{n+1} )$, $\alpha\in I_{k_{n+1}}$, are a quasi-partition of $N\times (r_1+\dots +2r_{n}, r_1+\dots +2r_{n}+2r_{n+1})$  made of big admissible sets. It is possible to choose $k_{n+1}$ and $r_{n+1}$ in such a way that $|B_N(0_N,C_N^{-1} \rd^{k_{n+1}})| 2r_{n+1} >\sigma$, so that $\mu(R^{n+1}_{\alpha})>\sigma$ for all $\alpha\in I_{k_{n+1}}$.

By iterating this process we get a quasi-partition of $N\times (-r_1,\infty)$ made of big admissible sets with measure greater than $\sigma$. By a similar procedure we get a quasi-partition of $N\times (-\infty,-r_1)$ made of big admissible sets with measure greater than $\sigma$, as required.
\end{proof}

Lemma \ref{lem:tagli} shows that we can iteratively consider children, children of children, children of children of children, ..., that is, \emph{descendants} of an admissible set and all these sets are admissible. In this way we can also define subsequent refinements of a quasi-partition of $G$ in admissible sets.  Namely, let $\mathcal{P}$ be a quasi-partition of $G$ in admissible sets and define $\desc^n(\mathcal{P})$ iteratively for all $n \in \NN$ as follows:
\[
\desc^0(\mathcal{P}) = \mathcal{P}, \qquad \desc^{n+1}(\mathcal{P}) = \bigcup_{R \in \desc^n(\mathcal{P})} \chil(R).
\]
Finally define $G_{\mathcal{P}} = \bigcap_{n \in \NN} \bigcup \desc^n(\mathcal{P})$ and $\desc(\mathcal{P}) = \bigcup_{n \in \NN} \desc^n(\mathcal{P})$. $\desc(\mathcal{P})$ is the set of descendants of elements of $\mathcal{P}$.

\begin{lem}\label{lem:partitionsplitting}
Let $\mathcal{P}$ be a quasi-partition of $G$ in admissible sets. Then the following hold:
\begin{enumerate}[label=(\roman*)]
\item\label{en:partitionsplitting1} For all $n \in \NN$, $\desc^n(\mathcal{P})$ is a quasi-partition of $G$ in admissible sets.
\item\label{en:partitionsplitting0} For all $R,R' \in \desc(\mathcal{P})$, either $R \cap R' = \emptyset$ or $R \subset R'$ or $R' \subset R$.
\item\label{en:partitionsplitting2} $G_{\mathcal{P}}$ has full measure in $G$.
\item\label{en:partitionsplitting3} For all $x \in G_{\mathcal{P}}$ and $n \in \NN$, there is a unique $R_x^n \in \desc^n(\mathcal{P})$ such that $x \in R_x^n$.
\item\label{en:partitionsplitting4} For all $x \in G_{\mathcal{P}}$ and all neighborhoods $U$ of $x$, there exists $n \in \NN$ such that $R_x^n \subset U$.
\end{enumerate}
\end{lem}
\begin{proof}
\mbox{\ref{en:partitionsplitting1}} is an immediate consequence of Lemma \ref{lem:tagli} and \ref{en:partitionsplitting2} follows because $G_\mathcal{P}$ is a countable intersection of sets of full measure in $G$.

About \ref{en:partitionsplitting0}, take $R \in \desc^n(\mathcal{P})$ and $R' \in \desc^{n'}(\mathcal{P})$ for some $n,n' \in \NN$. If $R \cap R' \neq \emptyset$, then necessarily $n \neq n'$. Suppose that $n < n'$. Then by construction $R'$ is descendant of exactly one element $R'' \in \desc^n(\mathcal{P})$. Consequently: either $R'' = R$ and therefore $R' \subset R$, or $R'' \cap R = \emptyset$ and then also $R' \cap R = \emptyset$.

As for \ref{en:partitionsplitting3}, since $x$ belongs to the union of $\desc^n(\mathcal{P})$ and $\desc^n(\mathcal{P})$ is a quasi-partition of $G$, clearly a set $R_x^n \in \desc^n(\mathcal{P})$ such that $x \in R_x^n$ exists and is unique. In fact from the construction it is clear that $R_x^{n+1}$ is a child of $R_x^n$ for all $n \in \NN$. In particular the sets $R_x^n$ for fixed $x$ form a decreasing sequence as $n$ grows in $\NN$ and, at each step, in the passage from $R_x^n = Q^k_\alpha \times (u_0-r,u_0+r)$ to its child $R_x^{n+1}$, either the first factor $Q^k_\alpha$ is replaced by one of its children $Q^{k-1}_\beta$, or the second factor $(u_0-r,u_0+r)$ is halved. In order to prove \ref{en:partitionsplitting4}, it will be then sufficient to show that each of these two alternatives does happen infinitely many times, i.e., that $R_x^n$ is strongly admissible for infinitely many $n$ and also that $R_x^n$ is not strongly admissible for infinitely many $n$: in fact, in this case, the diameter of both projections of $R_x^n$ onto the two factors $N$ and $A$ of $G$ tends to $0$ as $n \to \infty$.

By contradiction, suppose that, for all $n$ greater than some $n_0$, $R_x^n$ is strongly admissible. This means that, if $R_x^{n_0} = Q^k_\alpha \times (u_0-r,u_0+r)$, then $R_x^n$ has the form $Q^{k+n_0-n}_{\alpha_n} \times (u_0-r,u_0+r)$ for all $n \geq n_0$, where $\alpha_n \in I_{k+n_0-n}$. Since the $R_x^n$ are all admissible, \eqref{eq:adm} must hold when $k$ is replaced by $\ell$ for all integers $\ell \leq k$, while $u_0$ and $r$ remain the same, and when $\ell$ tends to $-\infty$ one obtains a contradiction. Similarly one obtains a contradiction by assuming that, for all $n \geq n_0$, $R_x^n$ is not strongly admissible: in this case one would have $\eta^k < 4(2^{-\ell} r) e^{8M} e^{u_0+r}$ for fixed $k$,$u_0$,$r$ and for all sufficiently large $\ell$, which is clearly impossible.
\end{proof}

For all quasi-partitions  $\mathcal{P}$  of $G$ in admissible sets, we define the maximal operator $M^{\mathcal{P}}$ as follows:
for all functions $f$ in $L_{\loc}^1(G)$ and $x \in G$,
\[
M^{\mathcal{P} }f(x)= \begin{dcases}
\sup_{\substack{R\in \desc(\mathcal{P})\\R \ni x}}\frac{1}{\mu({R})}\int_R|f|\di\mu & \text{if $x \in \bigcup \desc(\mathcal{P})$,}\\
0 & \text{otherwise}.
\end{dcases}
\]

\begin{prop}\label{prp:maximal}
Let $\mathcal{P}$ be a quasi-partition of $G$ in admissible sets.
\begin{enumerate}[label=(\roman*)]
\item\label{en:maximal1} $M^\mathcal{P} f$ is measurable for all $f \in L^1_\loc(G)$ and
\begin{equation}\label{eq:sublinear}
M^\mathcal{P} (\lambda f + \lambda' f') \leq |\lambda| M^\mathcal{P} f + |\lambda'| M^\mathcal{P} f'
\end{equation}
for all $\lambda,\lambda' \in \CC$ and $f,f' \in L^1_\loc(G)$.
\item\label{en:maximal2} The maximal operator $M^{\mathcal{P}}$ is of weak type $(1,1)$.
\item\label{en:maximal3} For all $f \in L^1_\loc(G)$, $|f| \leq M^{\mathcal{P}} f$ almost everywhere.
\end{enumerate}
\end{prop}
\begin{proof}
\mbox{\ref{en:maximal1}}. $M^\mathcal{P} f = \sup_{n \in \NN} M^\mathcal{P}_n f$, where
\[
M^\mathcal{P}_n f(x) = \begin{dcases}
\frac{1}{\mu({R_x^n})}\int_{R_x^n} |f|\di\mu & \text{if $x \in \bigcup \desc^n(\mathcal{P})$,}\\
0 & \text{otherwise},
\end{dcases}
\]
and the sets $R_x^n$ are defined as in Lemma \ref{lem:partitionsplitting}. Clearly the $M^\mathcal{P}_n f$ are measurable and consequently $M^\mathcal{P} f$ is measurable too. The inequality \eqref{eq:sublinear} is clear by the definition.

\ref{en:maximal2}. Let $f$ be in $L^1(G)$ and $\alpha>0$. Consider the set $\Omega_{\alpha}=\{M^{\mathcal{P}}f>\alpha \}$. For each point $x\in \Omega_{\alpha}$ let $R_x$ be the largest set (in the sense of inclusion) in $\desc(\mathcal{P})$ that contains $x$ such that the average of $|f|$ on $R_x$ is greater than $\alpha$. If $\mathcal{S} = \{R_x \tc x \in \Omega_\alpha \}$, then $\mathcal{S}$ is a partition of $\Omega_\alpha$ made of elements of $\desc(\mathcal{P})$. Thus,
\[
\mu\big(\Omega_{\alpha}\big)=\sum_{R \in \mathcal{S}} \mu(R) \leq \frac{1}{\alpha} \sum_{R \in \mathcal{S}} \int_{R}|f| \di\mu \leq \frac{\|f\|_{1}}{\alpha}.
\]

\ref{en:maximal3}. By \ref{en:maximal2} and standard arguments (cf.\ \cite[Theorem II.3.12]{SW} or \cite[Theorems 2.2 and 2.10]{duoandikoetxea_fourier_2001}) it is sufficient to consider the case where $f$ is continuous. In this case
\[
M^{\mathcal{P}} f(x) \geq \lim_{n \to \infty} \frac{1}{\mu(R_x^n)} \int_{R_x^n} |f| \di\mu = |f(x)|
\]
for all $x \in G_\mathcal{P}$, by Lemma \ref{lem:partitionsplitting}\ref{en:partitionsplitting4}, and $G_\mathcal{P}$ has full measure by Lemma \ref{lem:partitionsplitting}\ref{en:partitionsplitting2}. 
\end{proof}

Now we are able to construct the \CZ\ decomposition of an integrable function on $G$.

\begin{teo}\label{thm:CZ}
The space $(G,\varrho,\mu)$ with the family $\mathcal{R}$ of admissible sets is a CZ-space which satisfies condition \condC.
\end{teo}
\begin{proof}
By Proposition \ref{prp:proprieta} the family $\mathcal{R}$ of admissible sets in $G$ satisfies conditions \ref{en:CZadm1}-\ref{en:CZadm2} of Definition \ref{dfn:CZdef}.

Let now $f$ be in $L^1(G)$ and $\alpha > 0$. Our purpose is to construct a \CZ\ decomposition of $f$ at height $\alpha$. Let $\mathcal{P}$ be a quasi-partition of $G$ in big admissible sets whose measure is greater than $\frac{\|f\|_{1}}{\alpha}$  (it does exist by Lemma \ref{lem:partizionegrande}). For each $R$ in $\mathcal{P}$ we have that $\frac{1}{\mu({R})}\int_R|f|\di\mu < \alpha$.

Let $\mathcal{B} = \{ R \in \desc(\mathcal{P}) \tc \mu(R)^{-1} \int_R |f| \di\mu \geq \alpha\}$. We define the family $\mathcal{C}$ of the stopping sets as follows:
\[
\mathcal{C} = \{ R \in \mathcal{B} \tc R' \notin \mathcal{B} \text{ for all } R' \in \desc(\mathcal{P}) \text{ such that } R \subsetneq R'\}.
\]
By Lemma \ref{lem:partitionsplitting}\ref{en:partitionsplitting0} it is clear that the elements of $\mathcal{C}$ are pairwise disjoint. On the other hand $\bigcup \mathcal{C} = \bigcup \mathcal{B}$; therefore, if $\Omega$ is the complement of $\bigcup \mathcal{C}$ in $G$, then
\begin{equation}\label{eq:maximalbound}
M^{\mathcal{P}} f(x) \leq \alpha \qquad\text{for all } x \in \Omega.
\end{equation}
Further, for all $R \in \mathcal{C}$, it is $R \in \mathcal{B}$, hence $R \notin \mathcal{P}$ and consequently $R$ is the child of some $R' \in \desc(\mathcal{P}) \setminus \mathcal{B}$; therefore
\begin{equation}\label{eq:mediabound}
\alpha \leq \mu(R)^{-1} \int_R |f| \di\mu \leq C_2 \mu(R')^{-1} \int_{R'} |f| \di\mu < C_2 \alpha
\end{equation}
by Lemma \ref{lem:tagli}\ref{en:tagli3}.

Define
\[
g=\sum_{E \in \mathcal{C}} \Big(\frac{1}{\mu(E)}\int_{E}f\di\mu \Big)\,\chi_{E} +f \chi_{\Omega} \qquad\text{and}\qquad b_E=\Big(f-\frac{1}{\mu(E)}\int_{E}f\di\mu \Big)\,\chi_{E}
\]
for all $E \in \mathcal{C}$.
By \eqref{eq:mediabound} it follows that $|g|\leq C_2 \alpha$ on each set $E \in \mathcal{C}$. Moreover, by \eqref{eq:maximalbound} and Proposition \ref{prp:maximal}\ref{en:maximal3},
\[
|g(x)|=|f(x)|\leq \alpha \qquad\text{for a.a.\ } x\in \Omega.
\]
Each function $b_E$ is supported in $E$ and has average zero. Moreover
\[
\sum_{E \in \mathcal{C}} \|b_E\|_{1} \leq 2\,\sum_{E \in \mathcal{C}} \int_{E}|f|\di\mu \leq 2\,\|f\|_{1}.
\]
Finally, again by \eqref{eq:mediabound} and disjointness of $\mathcal{C}$,
\[
\sum_{E \in \mathcal{C}} \mu(E) \leq\frac{1}{\alpha}\,\sum_{E \in \mathcal{C}} \int_{E}|f|\di\mu\leq\frac{1}{\alpha}\,\,\|f\|_{1}\,.
\]

Thus $f=g+ \sum_{E \in \mathcal{C}} b_E$ is a \CZ\ decomposition of the function $f$ at height $\alpha$. The CZ-constant of the space is $\kappa_0=\max \{C_1,C_2,C^*\}$.

To conclude the proof of the theorem we shall construct a family of admissible sets $\mathcal R'$ which satisfies condition \condC. To do so, for all $k \in \ZZ^+$ define $r_k = \frac{k}{2M} \log \eta$. Then clearly $e^{2Mr_k}\leq \rd ^k<4e^{8Mr_k}$ and $r_k \to \infty$ as $k \to \infty$, so $r_k \geq r_0$ if $k \geq k_0$, say. Consequently, for all $k \geq k_0$ and $\alpha \in I_k$, the sets
\begin{equation}\label{eq:Rkm}
R^k_{\alpha}=Q^k_{\alpha}\times (-{r_k},r_k)
\end{equation}
are admissible. Set $\mathcal R'=\{ R^k_{\alpha} \tc k \geq k_0, \alpha \in I_k \}$. The following properties are satisfied:
\begin{enumerate}[label=(\roman*)]
\item If $R^k_{\alpha}\cap R^{\ell}_{\beta}\neq\emptyset$ and $k>\ell$, then $R^{\ell}_{\beta}\subset R^k_{\alpha}$.
\item If $R=Q_{\beta}^{\ell}\times (u_0-r,u_0+r)$ is an admissible set, then there exist $k \geq k_0$ and $\alpha\in I_k$ such that $R\subset R^k_{\alpha}$. Indeed, we may choose $k\geq \max\{\ell,k_0\}$ such that $(u_0-r,u_0+ r)\subset (-{r_k},r_k)$. In this case, there exists $\alpha\in I_{k}$ such that $Q_{\beta}^{\ell}\subset Q^k_{\alpha}$.
\end{enumerate}
Thus condition \condC\ is satisfied.
\end{proof}

Since by Theorem \ref{thm:CZ} the space $(G,\varrho,\mu)$ satisfies condition \condC, we can define a Hardy space $H^1(G)$ and a space $BMO(G)$ as in Definitions \ref{dfn:Hardy} and \ref{dfn:BMO}. By using the geometric properties of $(G,\varrho,\mu)$ and the properties of admissible sets, one can easily check that all the results obtained in \cite{V3} and \cite{LVY} for Hardy and $BMO$ spaces on $ax+b$-groups can be proved also in our setting, with only slight changes in their proofs (see, e.g., \cite{BL} for definition and discussion of the real and complex interpolation methods).

\begin{prop}\label{prp:H1BMOprop}
The following hold:
\begin{enumerate}[label=(\roman*)]
\item (John--Nirenberg inequality) there exist two positive constants $\gamma$ and $D$ such that for any $s>0$, $R\in\mathcal R$ and $g\in BMO(G)$,
\[
\mu\big(\{x\in R:~|g(x)-g_R|>s\,\|g\|_{BMO}\}\big)\leq D\,e^{-\gamma \,s}\,\mu({R});
\]
\item $\big(H^1(G),L^{2}(G)  \big)_{\theta,p}=L^p(G)\,,$ where $\theta\in (0,1)$, $\frac1p=1-\frac{\theta}{2}$ and $\big(\cdot, \cdot  \big)_{\theta,p}$ denotes the interpolation space obtained by the real method;
\item $\big(H^1(G),L^{2}(G)  \big)_{[\theta]}=L^p(G)\,,$ where $\theta\in (0,1)$, $\frac1p=1-\frac{\theta}{2}$ and $\big(\cdot, \cdot  \big)_{[\theta]}$ denotes the interpolation space obtained by the complex method;
\item $\big(L^{2}(G), BMO(G)  \big)_{\theta,p}=L^p(G)\,,$ where $\theta\in (0,1)$, $\frac1p=\frac{1-\theta}{2}$;
\item $\big(L^{2}(G) ,BMO(G) \big)_{[\theta]}=L^p(G)\,,$ where $\theta\in (0,1)$, $\frac1p=\frac{1-\theta}{2}$.
\end{enumerate}
\end{prop}

\section{The \texorpdfstring{sub-Laplacian $\Delta$,}{sub-Laplacian,} its heat kernel, and its spectral multipliers}\label{s:sub-Laplacian}

\subsection{The sub-Laplacian}
Let $\Delta$ be the sub-Laplacian defined in \eqref{eq:sub-Laplacian}. We recall now some well-known properties of $\Delta$, that are common to all left-invariant sub-Laplacians on Lie groups (see, e.g., \cite{varopoulos_analysis_1992}, \cite{martini_spectral_2011}, and references therein for further details).

Since the horizontal distribution on $G$ is bracket-generating, $\Delta$ is hypoelliptic \cite{hrmander_hypoelliptic_1967}. Moreover $\Delta$ is essentially self-adjoint and positive with respect to the right Haar measure; in fact, for all $f,g \in C^\infty_c(G)$,
\begin{equation}\label{eq:dirichlet}
\langle \Delta f, g \rangle = \sum_{j=0}^\fdim \langle \vfG_j f , \vfG_j g \rangle,
\end{equation}
where $\langle \cdot,\cdot \rangle$ denotes the inner product of $L^2(G)$.

In particular $\Delta$ extends uniquely to a positive self-adjoint operator on $L^2(G)$ and for all bounded Borel functions $F : \Rnon \to \CC$, the operator $F(\Delta)$ is a convolution operator with kernel $k_{F(\Delta)}$ (see \eqref{eq:convolutionkernel}). By means of the convolution formula, when $k_{F(\Delta)} \in L^1_\loc(G)$, we can interpret $F(\Delta)$ as an integral operator with integral kernel $K_{F(\Delta)}$ given by
\begin{equation}\label{eq:integralkernel}
K_{F(\Delta)}(x,y) =k_{F(\Delta)}(y^{-1}x)\,m(y) \qquad \text{for a.a. } x,y\in G.
\end{equation}

In the sequel we will often make use of some properties of differential equations associated with $\Delta$. First of all, we have finite propagation speed \cite{melrose_propagation_1986,cowling_subfinsler_2013} for solutions of the wave equation:
\[
\supp (\cos(t \sqrt{\Delta}) f) \subset \{ x \in G \tc \dist(x,\supp f) \leq t\}
\]
for all $f \in L^2(G)$ and all $t \geq 0$.

Moreover, since $\Delta$ is associated to the Dirichlet form \eqref{eq:dirichlet} and annihilates constants, the heat kernel $t \mapsto h_t = k_{e^{-t\Delta}}$ is a semigroup of probability measures on $G$ \cite{HU}. By hypoellipticity of $\partial_t + \Delta$, the distribution $(t,x) \mapsto h_t(x)$ is in fact a smooth function on $\Rpos \times G$ and from the above discussion it follows that
\[
h_t * h_{t'} = h_{t+t'}, \qquad h_t \geq 0, \qquad \| h_t \|_1 = 1
\]
(semigroup of probability measures) and
\[
h_t^* = h_t, \qquad h_t(x) \leq m(x)^{1/2}  h_t(0)
\]
(self-adjointness and positivity on $L^2$). It is also possible to obtain ``Gaussian-type'' estimates for $h_t$ and its left-invariant derivatives: for all $p \in [1,\infty]$, $\alpha = (\alpha_0,\dots,\alpha_\fdim) \in \NN^{1+\fdim}$ and all $b \geq 0$ there exist $C,\omega \geq 0$ such that
\begin{equation}\label{eq:smalltimeheatkernelestimates}
\| e^{b |\cdot|_\dist} \vfG^\alpha h_t \|_{p} \leq C t^{-(Q+1)/(2p')-|\alpha|/2} e^{\omega t},
\end{equation}
where $p' = p/(p-1)$, $\vfG^\alpha = \vfG_0^{\alpha_0} \cdots \vfG_\fdim^{\alpha_\fdim}$, and $|\alpha| = \alpha_0+\dots+\alpha_\fdim$ (see, e.g., \cite{varopoulos_analysis_1992}, \cite{ter_elst_weighted_1998}, or \cite[Theorem 2.3(f)]{martini_spectral_2011}). These estimates are however of little use for $t$ large.

\subsection{\texorpdfstring{$L^1$}{L1} gradient heat kernel estimates}\label{subs:heat}
The heat kernel $h_t$ associated to $\Delta$ can be expressed in terms of the heat kernel $h_t^N$ associated to the sub-Laplacian $\Delta^N=-\sum_{j=1}^\fdim \vfN_j^2$ on $N$ (see \cite[\S 3]{mustapha_multiplicateurs_1998} or \cite[\S 2]{gnewuch_differentiable_2006}):
\begin{equation}\label{eq:integral}
h_t(z,u) = \int_0^\infty \Psi_t(\xi) \, \exp\left(-\frac{\cosh u}{\xi}\right) h^N_{e^u \xi /2}(z) \di\xi,
\end{equation}
where
\begin{equation}\label{eq:defPsi}
\Psi_t(\xi) = \frac{\xi^{-2}}{\sqrt{4\pi^3 t}} \exp\left(\frac{\pi^2}{4t} \right) \int_0^\infty \sinh\theta \, \sin\frac{\pi\theta}{2t} \,\exp\left(-\frac{\theta^2}{4t} -\frac{\cosh \theta}{\xi}\right) \di\theta.
\end{equation}
This formula was used in the aforementioned works to obtain $L^1$-estimates for the heat kernel $h_t$ at complex times $t = 1+i\tau$, $\tau \in \RR$. Here we will show that the same formula can be used to obtain $L^1$-estimates for the horizontal gradient of the heat kernel $h_t$ at real times $t \in \RR$.

For a (smooth) function $f$ on $G$ we define the horizontal gradient $\nabla_H f(x) \in H_x G$ at $x \in G$ by
\[
g_x(\nabla_H f(x),v) = (\di f)_x(v)  \qquad\forall v \in H_x G,
\]
where $(\di f)_x$ is the differential of $f$ at $x$. It is easily seen that
\begin{equation}\label{eq:nablaH}
|\nabla_H f(x)|^2_g  = g_x(\nabla_H f(x),\nabla_H f(x)) = \sum_{j=0}^\fdim |\vfG_j f(x)|^2.
\end{equation}
In order to estimate the $L^1$-norm of $|\nabla_Hh_t|_g$ we shall repeatedly use the following technical lemma.
\begin{lem}\label{lem:innerintegral}
For all $\alpha,\theta \geq 0$,
\[
\int_\RR \int_0^\infty \frac{\cosh(\alpha u)}{\xi^{2+\alpha}} \,   \exp\left(-\frac{\cosh \theta+\cosh u}{\xi}\right) \di\xi \di u \leq C_\alpha \begin{cases}
e^{-\theta} &\text{if $\alpha > 0$,}\\
e^{-\theta} (1+\theta) &\text{if $\alpha = 0$.}
\end{cases}
\]
\end{lem}
\begin{proof}
The inner integral in $\xi$ is convergent and by rescaling it is equal to a constant times $(\cosh \theta + \cosh u)^{-1-\alpha} \cosh(\alpha u)$. Consequently the integral in $u$ is controlled by a constant times
\[
\int_0^\infty (e^\theta + e^u)^{-1-\alpha} \, e^{\alpha u}  \di u 
= e^{-\theta} \int_{e^{-\theta}}^\infty (1 + v)^{-1-\alpha} \, v^{\alpha-1} \di v
\]
and the conclusion follows.
\end{proof}

\begin{prop}\label{prp:gradientestimates}
There exists $C > 0$ such that
\[
\left\| \left|\nabla_H h_t\right|_g\right\|_{1} \leq C t^{-1/2}\qquad \forall t \in \Rpos.
\]
 \end{prop}
\begin{proof}
By \eqref{eq:nablaH} it suffices to show that
\[
\|\vfG_j h_t\|_{1} \leq C t^{-1/2}
\]
for all $j \in \{0,1,\dots,\fdim\}$ and $t \in \Rpos$. These estimates are already known in the case $t \leq 1$, see \eqref{eq:smalltimeheatkernelestimates}. Therefore in the rest of the proof we will assume that $t \geq 1$.

Note that, by homogeneity considerations, the corresponding estimates for the heat kernel on $N$ are easily shown to hold for all times. In fact
\begin{equation}\label{eq:n_heat_estimates}
\|h_s^N\|_{L^1(N)} = 1, \qquad \|\vfN_j h_s^N\|_{L^1(N)} = C_j s^{-1/2}
\end{equation}
for all $s \in \Rpos$ and $j\in\{1,\dots,\fdim\}$ (see, e.g., \cite[Proposition (1.75)]{folland_hardy_1982}). These equations, together with the formula \eqref{eq:integral}, will be the main ingredient of our proof.

We consider first the case $j>0$. Recall that $\vfG_j = e^{u} \vfN_j$. Then, by \eqref{eq:integral} and differentiation under the integral sign,
\[
\vfG_j h_t(z,u) = \int_0^\infty \Psi_t(\xi) \, \exp\left(-\frac{\cosh u}{\xi}\right) e^{u} \vfN_j h^N_{e^u \xi /2}(z) \di\xi.
\]
Therefore, by \eqref{eq:n_heat_estimates},
\[
\|\vfG_j h_t\|_1 \lesssim \int_\RR \int_0^\infty |\Psi_t(\xi)| \, \frac{e^{u/2}}{\xi^{1/2}} \exp\left(-\frac{\cosh u}{\xi}\right) \di\xi \di u.
\]
Since $t\geq 1$, by \eqref{eq:defPsi} the above integral is controlled by a constant times
\begin{multline*}
t^{-1/2} \int_0^\infty \sinh\theta \, \left|\sin\frac{\pi\theta}{2t} \right| \,\exp\left(-\frac{\theta^2}{4t}\right) \\
\times \int_\RR \int_0^\infty \frac{e^{u/2}}{\xi^{2+1/2}} \exp\left(-\frac{\cosh \theta+\cosh u}{\xi}\right) \di\xi \di u \di\theta.
\end{multline*}
By applying Lemma \ref{lem:innerintegral} (with $\alpha=1/2$), the integral in $u$ is controlled by a constant times $e^{-\theta}$, hence
\[
\|\vfG_j h_t\|_1 \lesssim t^{-1/2} \int_0^\infty \frac{\sinh\theta}{e^\theta} \, \frac{\theta}{t}  \,\exp\left(-\frac{\theta^2}{4t}\right) \di \theta \lesssim t^{-1/2}.
\]

For $j=0$ we have instead, again by \eqref{eq:integral},
\[\begin{split}
\vfG_0 h_t(z,u) &= -\int_0^\infty \Psi_t(\xi) \, \frac{\sinh u}{\xi} \, \exp\left(-\frac{\cosh u}{\xi}\right) h^N_{e^u \xi /2}(z) \di \xi \\
&\qquad + \int_0^\infty \Psi_t(\xi) \, \exp\left(-\frac{\cosh u}{\xi}\right) \frac{\partial}{\partial u} [h^N_{e^u \xi /2}(z)] \di \xi = I_1 + I_2.
\end{split}\]
The $L^1$-norm of the first summand $I_1$ can be controlled analogously as above (here the first identity in \eqref{eq:n_heat_estimates} is used and Lemma \ref{lem:innerintegral} is applied with $\alpha=1$). For the second term $I_2$, instead, we need some further manipulation.

Note that $\partial_u [h^N_{e^u \xi /2}(z)] = \xi \partial_\xi [h^N_{e^u \xi /2}(z)]$. Hence, by integration by parts,
\[\begin{split}
I_2 
&= -\int_0^\infty \frac{\partial}{\partial \xi}[\xi \, \Psi_t(\xi)] \, \exp\left(-\frac{\cosh u}{\xi}\right) \, h^N_{e^u \xi /2}(z) \di \xi \\
&\qquad-\int_0^\infty \Psi_t(\xi) \, \frac{\cosh u}{\xi} \exp\left(-\frac{\cosh u}{\xi}\right) \, h^N_{e^u \xi /2}(z) \di \xi = I_3 + I_4.
\end{split}\]
The term $I_4$ can be controlled in the same way as $I_1$.
As for $I_3$, we observe, by \eqref{eq:defPsi}, that
\[\begin{split}
\frac{\partial}{\partial \xi}&[\xi \, \Psi_t(\xi)] \\
&= \frac{\exp\left(\frac{\pi^2}{4t} \right)}{\xi^{2}\sqrt{4\pi^3 t}}  \int_0^\infty \sinh\theta \, \sin\frac{\pi\theta}{2t} \,\exp\left(-\frac{\theta^2}{4t} -\frac{\cosh \theta}{\xi}\right) \left( \frac{\cosh \theta}{\xi} - 1\right)\di \theta \\
&= -\frac{\exp\left(\frac{\pi^2}{4t} \right)}{\xi^{2}\sqrt{4\pi^3 t}} \int_0^\infty \cosh\theta \, \sin\frac{\pi\theta}{2t} \,\exp\left(-\frac{\theta^2}{4t}\right) \frac{\partial}{\partial \theta} \left[\exp\left(-\frac{\cosh \theta}{\xi}\right)\right] \di\theta - \Psi_t(\xi) \\
&= \frac{\exp\left(\frac{\pi^2}{4t} \right)}{\xi^{2}\sqrt{4\pi^3 t}} \int_0^\infty \frac{\partial}{\partial \theta} \left[\cosh\theta \, \sin\frac{\pi\theta}{2t} \,\exp\left(-\frac{\theta^2}{4t}\right)\right] \exp\left(-\frac{\cosh \theta}{\xi}\right) \di\theta - \Psi_t(\xi) \\
&= \frac{\exp\left(\frac{\pi^2}{4t} \right)}{2t \xi^{2}\sqrt{4\pi^3 t}} \int_0^\infty \cosh\theta \, \left[\pi \cos\frac{\pi\theta}{2t} -  \theta \sin\frac{\pi\theta}{2t}\right] \exp\left(-\frac{\theta^2}{4t}\right) \exp\left(-\frac{\cosh \theta}{\xi}\right) \di\theta.
\end{split}\]
Consequently, since $t \geq 1$, by \eqref{eq:n_heat_estimates} the $L^1$-norm of $I_3$ is bounded by a constant times
\begin{multline*}
t^{-3/2} \int_0^\infty \cosh\theta \, \left|\pi \cos\frac{\pi\theta}{2t} -  \theta \sin\frac{\pi\theta}{2t}\right| \exp\left(-\frac{\theta^2}{4t}\right) \\
\times \int_\RR \int_0^\infty \xi^{-2}  \exp\left(-\frac{\cosh \theta + \cosh u}{\xi}\right) \di\xi \di u \di\theta.
\end{multline*}
By applying Lemma \ref{lem:innerintegral} (with $\alpha=0$), the integral in $u$ is controlled by a constant times $e^{-\theta} (1+\theta)$, hence
\[
\|I_3\|_{1} \lesssim t^{-3/2} \int_0^\infty \frac{\cosh\theta}{e^\theta} (1 + \theta^2/t) \, (1+\theta) \, \exp\left(-\frac{\theta^2}{4t}\right) \di\theta \lesssim t^{-1/2}
\]
and we are done.
\end{proof}

\subsection{The Plancherel measure and weighted \texorpdfstring{$L^2$}{L2}-estimates}\label{subs:plancherel}

By abstract nonsense (see, e.g., \cite[Theorem 3.10]{martini_spectral_2011} for a quite general statement) one can show that there exists a Plancherel measure associated with $\Delta$, i.e., a positive Borel measure $\sigma_\Delta$ on $\Rnon$, whose support is the $L^2(G)$-spectrum of $\Delta$, such that
\begin{equation}\label{eq:planchereldef}
\| k_{F(\Delta)} \|_{2}^2 = \int_{\Rnon} |F(\lambda)|^2 \di\sigma_\Delta(\lambda)
\end{equation}
for all bounded Borel functions $F : \Rnon \to \CC$.

In the case $N$ is abelian, $G$ is a real hyperbolic space and the Plancherel measure $\sigma_\Delta$ can be explicitly computed via spherical analysis (cf.\ \cite{cowling_spectral_1994,cowling_estimates_1993}); namely there exists $c_\Delta \in \Rpos$ such that
\begin{equation}\label{eq:plancherel}
\int_\Rnon F(\lambda) \di \sigma_\Delta(\lambda) = c_\Delta \int_{\RR} F(s^2) \, |\hc_Q(s)|^{-2} \di s,
\end{equation}
where $\hc_Q$ is the Harish-Chandra function for the $(Q+1)$-dimensional real hyperbolic space (see, e.g., \cite[Theorem IV.6.14]{HE}), so
\[
|\hc_Q(s)|^{-2} \sim \begin{cases}
|s|^2 &\text{for $|s|$ small,}\\
|s|^Q &\text{for $|s|$ large.}
\end{cases}
\]

In the case $N$ is nonabelian, spherical analysis can no longer be directly applied to the functional calculus of $\Delta$. Nevertheless, as we are going to show, the above formula for the Plancherel measure remains valid.

Let $\mathcal{J}$ be the set of functions $\Rnon \to \CC$ that are finite linear combinations of decaying exponentials $\lambda \mapsto e^{-t\lambda}$ ($t \in \Rpos$). Note that $\mathcal{J}$ is uniformly dense in $C_0(\Rnon)$ by the Stone--Weierstrass theorem. The following fundamental observation is in \cite[Lemma (1.10)]{HEB}; here we provide an alternative proof using \eqref{eq:integral}.

\begin{prop}\label{prp:g_n_calculi}
For all $F \in \mathcal{J}$ and all $u \in \RR$, there exists a bounded Borel function $M_{F,u} : \Rnon \to \CC$ such that
\begin{equation}\label{eq:kernelsections}
k_{F(\Delta)}(\cdot,u) = k_{M_{F,u}(\Delta^N)}
\end{equation}
and $M_{F,u}$ does not depend on the stratified group $N$ or the sub-Laplacian $\Delta^N$.
\end{prop}
\begin{proof}
By linearity it is sufficient to consider the case where $F(\lambda) = e^{-t\lambda}$. However in this case, if we set
\[
M_{F,u}(\lambda) = \int_0^\infty \Psi_t(\xi) \, \exp\left(-\frac{\cosh u}{\xi}\right) \exp(-e^u \xi \lambda/2) \di\xi,
\]
then \eqref{eq:kernelsections} follows from the formula \eqref{eq:integral} for the heat kernel $h_t = k_{e^{-t\Delta}}$. Note that the above expression for $M_{F,u}$ depends only on $t$ and $u$ and does not depend on the particular choice of $N$ or $\Delta^N$.
\end{proof}

Let $\Delta^{\RR^Q}$ be the Laplacian on $\RR^Q$ and $\tilde \Delta = -\partial_u^2 + e^{2u} \Delta^{\RR^Q}$ be the corresponding Laplacian on $\tilde G = \RR^Q \rtimes \RR$. Homogeneity and finite propagation speed properties of $\Delta^N$ and $\Delta^{\RR^Q}$ yield the following result.

\begin{prop}\label{prp:n_r_calculi}
For all $a \geq 0$ there is $C \in \Rpos$ such that, for all bounded Borel functions $f : \RR \to \CC$,
\[
\int_N |z|_N^a |k_{f(\Delta^N)}(z)|^2 \di z \leq C \int_{\RR^Q} |z|_{\RR^Q}^a |k_{f(\Delta^{\RR^Q})}(z)|^2 \di z,
\]
with equality if $a = 0$.
\end{prop}
\begin{proof}
See \cite[formula (3) and Lemme 2]{sikora_multiplicateurs_1992}.
\end{proof}

\begin{coro}\label{cor:n_hyp_calculi}
For all $a \geq 0$ there is $C \in \Rpos$ such that, for all bounded Borel functions $F : \RR \to \CC$,
\begin{equation}\label{eq:weightedinequality}
\int_G |z|_N^a |k_{F(\Delta)}(z,u)|^2 \di\mu(z,u) \leq C \int_{\tilde G} |z|_{\RR^Q}^a |k_{F(\tilde \Delta )}(z,u)|^2 \di\tilde \mu(z,u),
\end{equation}
with equality if $a = 0$, where $\di\tilde \mu(z,u)=\di z\di u$ is the right Haar measure on $\tilde G$.
\end{coro}
\begin{proof}
In the case $F \in \mathcal{J}$ the above inequality (or equality if $a=0$) follows immediately by combining Propositions \ref{prp:g_n_calculi} and \ref{prp:n_r_calculi}. The general case is then given by density.
\end{proof}

By comparing the case $a=0$ of Corollary \ref{cor:n_hyp_calculi} with the characterization \eqref{eq:planchereldef} of the Plancherel measure we obtain immediately the following result.

\begin{coro}\label{cor:plancherel}
For an arbitrary stratified group $N$ of homogeneous dimension $Q$, the Plancherel measure $\sigma_\Delta$ is given by \eqref{eq:plancherel} for some constant $c_\Delta \in \Rpos$. In particular the $L^2$-spectrum of $\Delta$ is $\Rnon$ and, for all Borel functions $F : \RR \to \CC$,
\[
\|k_{F(\sqrt{\Delta})}\|_2 \sim \left(\int_0^\infty |F(\lambda)|^2 (\lambda^3 + \lambda^{Q+1}) \,\frac{\di \lambda}{\lambda}\right)^{1/2}.
\]
\end{coro}

\section{The multiplier theorem}\label{s:theorem}

In this section we prove Theorem \ref{thm:moltiplicatori}. To do so, we need some preliminary estimates of the $L^1$-norm of the convolution kernels of spectral multipliers of $\Delta$.

\begin{prop}\label{prp:l1l2estimate}
There exists a positive constant $C$ such that, for all $r > 0$ and all even bounded Borel functions $F : \RR \to \CC$ whose Fourier transform $\hat F$ is supported in $[-r,r]$,
\[
\|k_{F(\sqrt{\Delta})}\|_{1} \leq C \min\{r^{(Q+1)/2}, r^{3/2}\} \|k_{F(\sqrt{\Delta})}\|_{2}.
\]
\end{prop}
\begin{proof}
Note that, since $\Delta$ satisfies finite propagation speed, $\supp k_{F(\sqrt{\Delta})} \subset \overline{B_\dist(0,r)}$ (see, e.g., \cite[Lemma 1.2]{cowling_spectral_2001}). Then, if $r \leq 1$, by H\"older's inequality and \eqref{eq:volumeasymptotics},
\[
\|k_{F(\sqrt{\Delta})}\|_{1} \lesssim r^{(Q+1)/2} \| k_{F(\sqrt{\Delta})} \|_{2}
\]
and we are done.

If instead $r \geq 1$, then, by H\"older's inequality and \eqref{eq:inverseweight},
\begin{equation}\label{eq:weighted_hoelder}
\|k_{F(\sqrt{\Delta})}\|_{1} \lesssim r \left( \| k_{F(\sqrt{\Delta})} \|_{2} + \| k_{F(\sqrt{\Delta})} \, w^{1/2}\|_{2} \right),
\end{equation}
where the weight $w$ is given by $w(z,u) = |z|_N^Q$. Therefore, by applying Corollary~\ref{cor:n_hyp_calculi} with $a=Q$, we have that
\begin{equation}\label{eq:relation_n_hyp}
\| k_{F(\sqrt{\Delta})} \, w^{1/2}\|_{2} \lesssim \| k_{F(\sqrt{\tilde\Delta})} \, \tilde w^{1/2}\|_{L^2(\tilde G)}
\end{equation}
where $\tilde w$ is the analogous weight on $\tilde G = \RR^Q \rtimes \RR$.
By spherical analysis on real hyperbolic spaces,
if $\tilde m$ is the modular function on $\tilde G$,
 then $k_{F(\sqrt{\tilde\Delta})} = \tilde m^{1/2} \phi_F$ for some radial function $\phi_F$ on $\tilde G$ (see, e.g., \cite[Proposition 1.2]{cowling_spectral_1994} and \cite[p.\ 148]{A1}). Moreover, if $\tilde\dist$ denotes the left-invariant Riemannian distance on $\tilde G$, then $\supp \phi_F = \supp k_{F(\sqrt{\tilde\Delta})} \subset \overline{B_{\tilde\dist}(0,r)}$, because $\tilde\Delta$ satisfies finite propagation speed too.
We can then apply \eqref{eq:weightedradialcomparison} and \eqref{eq:radialdensity} to obtain that
\begin{equation}\label{eq:steps}
\begin{split}
\| k_{F(\sqrt{\tilde\Delta})} \, \tilde w^{1/2}\|_{L^2(\tilde G)} &= \|  \tilde w^{1/2}\,   \tilde m^{1/2}  \phi_F  \|_{L^{2}(\tilde G)} \\
&\lesssim \|\phi_F \, \tilde\dist(\cdot,0_{\tilde G})^{1/2} \|_{L^{ 2}(\tilde G)} \\
&\lesssim r^{1/2} \|\phi_F \|_{L^{2}(\tilde G)} \\
&=r^{1/2} \|\phi_F \,\tilde m^{1/2} \|_{L^{2}(\tilde G)} \\
&=        r^{1/2} \| k_{F(\sqrt{\tilde\Delta})} \|_{L^2(\tilde G)} \\&\sim r^{1/2} \| k_{F(\sqrt{\Delta})} \|_{2},
\end{split}
\end{equation}
where the last step is given by Corollary \ref{cor:n_hyp_calculi} in the case $a=0$. The conclusion follows by combining \eqref{eq:relation_n_hyp} and \eqref{eq:steps} and plugging the resulting inequality into \eqref{eq:weighted_hoelder}.
\end{proof}

The next lemma shows that every function $f$ supported in $[1/2,2]$ may be written as sum of functions whose Fourier transforms have compact support.

\begin{lem}\label{lem:decomp}
Let $f \in L^2(\RR)$ be even and supported in $[-2,2]$. Then there exist even functions $f_{\ell}$, $\ell \in \NN$, such that
\begin{enumerate}[label=(\roman*)]
\item $f=\sum_{\ell=0}^{\infty} f_{\ell}$;
\item $\supp \hat{f}_{\ell} \subset [-2^{\ell},2^{\ell}]$;
\item for all $\alpha,\beta,s \in \Rnon$,
\[
\int_0^{\infty}|f_{\ell}(\lambda)|^2\,(\lambda^{\alpha}+\lambda^{\beta})\di\lambda\leq C_{\alpha,\beta,s}\, 2^{-2s\ell}\,\|f\|^2_{H^s(\RR)}.
\]
\end{enumerate}
Let $f_t$ denote the dilated of $f$ defined by $f_t=f(t\cdot)$. Then
\begin{enumerate}[label=(\roman*')]
\item $f_t=\sum_{\ell} f_{\ell,t}$, where $f_{\ell,t}=f_{\ell}(t\cdot)\,;$
\item $\supp \hat{f}_{\ell,t} \subset [-2^{\ell}t,2^{\ell}t]\,;$
\item for all $\alpha,\beta,s \in \Rnon$,
\[
\int_0^{\infty}|f_{\ell,t}(\lambda)|^2\,(\lambda^{\alpha}+\lambda^{\beta})\di\lambda\leq C_{\alpha,\beta,s} \max \{t^{-(\alpha+1)},t^{-(\beta+1)} \}
\,2^{-2s\ell}\,\|f\|^2_{H^s(\RR)}.
\]
\end{enumerate}
\end{lem}
\begin{proof}
See \cite[Lemma (1.3)]{HEB}.
\end{proof}

\begin{prop}\label{prp:stimaA}
Let $F \in L^2(\RR)$ be supported in $[-4,4]$. Then the estimate
\begin{equation}\label{eq:stimaA}
\sup_{y \in G} \int_G|K_{F(t\Delta)}(x,y)|\,\big(1+t^{-1/2} \dist(x,y)\big)^{\varepsilon}\, \di\mu(x) \leq  C_{s,\varepsilon} \|F\|_{H^s}
\end{equation}
holds for all $\varepsilon \in \Rnon$ and $s,t \in \Rpos$ satisfying one of the following conditions:
\begin{itemize}
\item $t \geq 1$ and $s > 3/2+\varepsilon$;
\item $t \leq 1$ and $s > (Q+1)/2+\varepsilon$.
\end{itemize}
\end{prop}
\begin{proof}
First we observe that, for all $y \in G$, by \eqref{eq:integralkernel} and the left-invariance of the metric $\dist$,
\[\begin{split}
\int_G &|K_{F(t\Delta)}(x,y)|\,\big(1+t^{-1/2} \dist(x,y)\big)^{\varepsilon} \di\mu(x)\\
&=\int_G |k_{F(t\Delta)}(y^{-1}\,x)|\,m(y)\,\big(1+t^{-1/2} \dist(y^{-1}\,x,0_G)\big)^{\varepsilon} \, \di\mu(x)  \\
&=\int_G|k_{F(t\Delta)}(x)|\,\big(1+t^{-1/2} \dist(x,0_G)\big)^{\varepsilon} \di\mu(x) .
\end{split}\]

Define $f(\lambda)=F(\lambda^2)$ for all $\lambda\in \RR$. The function $f$ is even and supported in $[-2,2]$, and $F(t\Delta) = f(t^{1/2}\sqrt{\Delta})$ for all $t \in \Rpos$. Moreover
\begin{equation}\label{eq:sobolevestimate}
\|f\|_{H^s} \lesssim \|F\|_{H^s}.
\end{equation}

Let $f = \sum_{\ell =0}^\infty f_\ell$ be the decomposition given by Lemma \ref{lem:decomp}. Since $f(t^{1/2}\cdot)=\sum_{\ell} f_{\ell,t^{1/2}}$ and $\supp\hat{f}_{\ell,t^{1/2}} \subset [-2^{\ell}t^{1/2}, 2^{\ell}t^{1/2}]$, we can apply Proposition \ref{prp:l1l2estimate} to each function $f_{\ell,t^{1/2}}$ and sum these estimates up.
Namely, by finite propagation speed, Proposition \ref{prp:l1l2estimate}, Corollary~\ref{cor:plancherel}, Lemma \ref{lem:decomp}(iii'), and \eqref{eq:sobolevestimate},
\[\begin{split}
\int_G &|k_{f_{\ell,t^{1/2}}(\sqrt{\Delta})}(x)|\,\big(1+t^{-1/2}\dist(x,0_G)\big)^{\varepsilon} \di\mu(x)\\
&\lesssim (1+t^{-1/2}2^{\ell}t^{1/2})^{\varepsilon} \|k_{f_{\ell,t^{1/2}}(\sqrt{\Delta})}\|_1 \\
&\lesssim 2^{\ell\varepsilon}\,\min\{(2^{\ell}t^{1/2})^{(Q+1)/2},(2^{\ell}t^{1/2})^{3/2}\}\,\|k_{f_{\ell,t^{1/2}}(\sqrt{\Delta})}\|_2 \\
&\lesssim 2^{\ell\varepsilon}\,\min\{(2^{\ell}t^{1/2})^{(Q+1)/2},(2^{\ell}t^{1/2})^{3/2}\}\,\Big(\int_0^{\infty}|f_{\ell,t^{1/2}}(\lambda)|^2\,(\lambda^2+\lambda^Q)\di \lambda\Big)^{1/2}\\
&\lesssim 2^{\ell\varepsilon}\,\min\{(2^{\ell}t^{1/2})^{(Q+1)/2},(2^{\ell}t^{1/2})^{3/2}\}\, \max\{t^{-3/4},t^{-(Q+1)/4}\} \, 2^{-\ell s} \, \|F\|_{H^s}.
\end{split}\]

In the case $t \geq 1$, it is then
\[
\int_G  |k_{F(t\Delta)}(x)|\,\big(1+t^{-1/2} \dist(x,0_G)\big)^{\varepsilon}\di\mu(x) \leq C_s  \|F\|_{H^s} \sum_{\ell \geq 0} 2^{\ell(\varepsilon+3/2-s)},
\]
and the series on the right-hand side converges since $s > 3/2+\varepsilon$.

In the case $t \leq 1$, instead, it is
\begin{multline*}
\int_G  |k_{F(t\Delta)}(x)|\,\big(1+t^{-1/2} \dist(x,0_G)\big)^{\varepsilon}\di\mu(x) \\
\lesssim  \|F\|_{H^s} \left( t^{3/4-(Q+1)/4} \sum_{\ell \tc 2^\ell \geq t^{-1/2}} 2^{\ell(\varepsilon+3/2-s)} + \sum_{\ell \tc 2^\ell < t^{-1/2}} 2^{\ell\big(\varepsilon+(Q+1)/2-s\big)} \right),
\end{multline*}
and the term in parentheses is finite and bounded above uniformly in $t$ since $s > (Q+1)/2+\varepsilon$.
\end{proof}

We denote by $R_y$ the right translation operator defined by
\[
R_y f(x)=f(xy)
\]
for all $f : G \to \CC$ and $x,y\in G$.

\begin{lem}\label{lem:srgradient}
For all $f \in L^1(G)$ and $y,z \in G$,
\[
\|R_y f - R_z f\|_1 \leq \dist(y,z) \left\| \left|\nabla_{H} f\right|_g  \right\|_1.
\]
\end{lem}
\begin{proof}
The proof of \cite[Lemma VIII.1.1]{varopoulos_analysis_1992} applies also to non-unimodular groups.
\end{proof}

\begin{prop}
Let $F \in L^2(\RR)$ be supported in $[-4,4]$. Then the estimate
\begin{equation}\label{eq:stimaB}
\int_G|K_{F(t\Delta)}(x,y)-K_{F(t\Delta)}(x,z)| \, \di\mu(x) \leq  C_{s} \, t^{-1/2} \dist(y,z) \, \|F\|_{H^s}
\end{equation}
holds for all $y,z \in G$ and $s,t \in \Rpos$ satisfying one of the following conditions:
\begin{itemize}
\item $t \geq 1$ and $s > 3/2$;
\item $t \leq 1$ and $s > (Q+1)/2$.
\end{itemize}
\end{prop}
\begin{proof}
By splitting $F$ into its real and imaginary parts, it is not restrictive to assume that $F$ is real-valued. In particular the operator $F(t\Delta)$ is self-adjoint and
\[\begin{split}
\int_G &|K_{F(t\Delta)}(x,y)-K_{F(t\Delta)}(x,z)| \, \di\mu(x) \\
&= \int_G|K_{F(t\Delta)}(y,x)-K_{F(t\Delta)}(z,x)| \, \di\mu(x) \\
&=\int_G| k_{F(t\Delta)}(x^{-1}y)-k_{F(t\Delta)}(x^{-1}z)   |\,m(x)\di\mu(x)\\
&=\int_G| k_{F(t\Delta)}(xy)-k_{F(t\Delta)}(xz)   |\di\mu(x)\\
&= \|R_y k_{F(t\Delta)} - R_z k_{F(t\Delta)} \|_1.
\end{split}\]
Define $\phi(\lambda)=F(\lambda)\,e^{-\lambda}$ for all $\lambda\in \RR$. Then
\[
k_{F(t\Delta)}=k_{\phi(t\Delta)}* h_t
\]
and, by Young's inequality,
\[
\|R_y k_{F(t\Delta)} - R_z k_{F(t\Delta)} \|_1 \leq \|k_{\phi(t\Delta)}\|_1 \|R_y h_t - R_z h_t \|_1.
\]
Note now that, under our assumptions on $t$ and $s$, by Proposition \ref{prp:stimaA} it follows that
\[
\|k_{\phi(t\Delta)}\|_1 \lesssim \|\phi\|_{H^s} \lesssim \|F\|_{H^s}.
\]
On the other hand, by Lemma \ref{lem:srgradient} and Proposition \ref{prp:gradientestimates},
\[
\|R_y h_t - R_z h_t \|_1 \leq \dist(y,z) \, \left\|\left| \nabla_{H} h_t \right|_g \right\|_1  \lesssim t^{-1/2} \dist(y,z)
\]
and the conclusion follows.
\end{proof}

We can finally prove our main result.

\begin{proof}[Proof of Theorem \ref{thm:moltiplicatori}]
Choose $\varepsilon > 0$ such that $s_0>\frac{3}{2}+\varepsilon$ and $s_{\infty}>\frac{Q+1}{2}+\varepsilon$.
Let $F$ be as in the statement of the theorem. It is not restrictive to assume that $F$ is real-valued, so $F(\Delta)$ is self-adjoint. Define
\[
F_j(\lambda)=F(2^j\lambda)\,\psi(\lambda)\qquad\forall j\in\ZZ \quad \forall\lambda\in\Rpos,
\]
where $\psi$ is as in \eqref{eq:sumpsi}. Then
\[
F(\Delta)=\sum_{j\in\ZZ} F_j(2^{-j}\Delta)
\]
in the sense of strong convergence of operators on $L^2(G)$, because the $L^2$-spectrum of $\Delta$ is $\Rnon$ and $\{0\}$ has null spectral measure. Since each function $F_j$ is supported in $[1/4,4]$ we may apply estimates (\ref{eq:stimaA}) and (\ref{eq:stimaB}) to $F_j$ and $t=2^{-j}$, to obtain that
\begin{equation}\label{eq:stimaAD}
\sup_{y\in G} \int_G|K_{F_j(2^{-j}\Delta)}(x,y)|\big(1+2^{j/2} \dist(x,y)\big)^{\varepsilon}  \di\mu(x) \lesssim  \begin{cases}
\|F\|_{0,s_0}& \forall~j\leq 0\\
\|F\|_{\infty,s_{\infty}}& \forall~j>0,
\end{cases}
\end{equation}
and, for all $y,z \in G$,
\begin{equation}\label{eq:stimaBD}
\int_G |K_{F_j(2^{-j}\Delta)}(x,y)-K_{F_j(2^{-j}\Delta)}(x,z)|\di\mu(x)
\lesssim \begin{cases}
2^{j/2}\dist(y,z) \,\|F\|_{0,s_0}& \forall~j\leq 0\\
2^{j/2}\dist(y,z) \,\|F\|_{\infty,s_{\infty}}& \forall~j>0.
\end{cases}
\end{equation}
Then the operator $F(\Delta)$ satisfies the hypotheses of Theorem \ref{thm:Teolim} and consequently it is of weak type $(1,1)$, bounded on $L^p(G)$ for all $p \in (1,2]$ and, by duality, for all $p \in [2,\infty)$. By Theorem \ref{thm:TeolimH1} it follows that $F(\Delta)$ is also bounded from $H^1(G)$ to $L^1(G)$ and a duality argument gives the boundedness from $L^{\infty}(G)$ to $BMO(G)$.
\end{proof}

\acknowledgments

\end{document}